\documentclass{amsart}
\usepackage{amsmath}
\usepackage{amsfonts}
\usepackage{amssymb}
\usepackage{epsfig}
\newtheorem{theorem}{Theorem}[section]

\newtheorem{lm}[theorem]{Lemma}

\newtheorem{cor}[theorem]{Corollary}

\newtheorem{rem}[theorem]{Remark}
\newtheorem{pr}[theorem]{Proposition}
\newtheorem{example}[theorem]{Example}

\usepackage{color}

\begin{document}

\title{On dimension theory of supermodules, super-rings and superschemes}
\author{}
\address{}
\email{}
\author{A. N. Zubkov}
\address{Department of Mathematical Science, UAEU, Al-Ain, United Arab Emirates; Sobolev Institute of Mathematics, Omsk Branch, Pevtzova 13, 644043 Omsk, Russia}
\email{a.zubkov@yahoo.com}
\author{P.S.Kolesnikov}
\address{Sobolev Institute of Mathematics, Akad. Koptyug prosp., 4, 630090 Novosibirsk, Russia}
\email{pavelsk@math.nsc.ru}
\begin{abstract}
We introduce the notion of Krull super-dimension of supermodules over certain super-commutative Noetherian super-rings.
We investigate how this notion relates to the notion of odd regular sequence from \cite{sm} and how it behaves with respect to the transition to the graded and bigraded supermodules and super-rings associated with the original ones. 
We also apply these results to the super-dimension theory of superschemes of finite type and their morphisms. 	
\end{abstract}
\maketitle

\section*{introduction}

The notion of Krull dimension plays crucial role in the theory of commutative rings and in the algebraic geometry. Recall that the Krull dimension $\mathrm{Kdim}(R)$ of a commutative ring $R$ is the supremum of heights of the prime ideals of $R$. If $\mathrm{Kdim}(R)<\infty$,
then $\mathrm{Kdim}(R)$ is just the largest nonnegative integer $n$ such that there is a chain $\mathfrak{p}_0\subseteq \ldots \subseteq\mathfrak{p}_n$ of pairwise different prime ideals. Geometrically, $\mathrm{Kdim}(R)$ is nothing else but the dimension of the affine scheme $\mathrm{Spec}(R)$ (cf. \cite{hart}, Chapter II, Example 3.2.7).  

In what follows all super-rings are supposed to be super-commutative, unless stated otherwise. If $R$ is a super-ring, then each prime superideal $\mathfrak{P}$ of $R$ has a form $\mathfrak{p}\oplus R_1$, where $\mathfrak{p}=\mathfrak{P}_0$ is a prime ideal of the ring $R_0$. In other words, the prime spectrum of $R$ coincides with the prime spectrum of $R_0$, and it does not detect the \emph{odd dimension}
of the affine superscheme $\mathrm{SSpec}(R)$. Therefore, a relevant notion of Krull super-dimension of a super-ring has to include an \emph{odd component} also. Such a notion has been recently suggested in \cite{zubmas}. 
Let $R$ be a super-ring with $\mathrm{Kdim}(R_0)$ to be finite. Then a collection of odd elements $y_1, \ldots, y_s$ is called a \emph{system of odd parameters} of $R$ if 
\[\mathrm{Kdim}(R_0/\mathrm{Ann}_{R_0}(y_1\ldots y_s))=\mathrm{Kdim}(R_0).\]
In other words, the elements $y_1, \ldots, y_s$ form a system of odd parameters of $R$ if there is a longest prime chain
$\mathfrak{p}_0\subseteq \ldots \subseteq\mathfrak{p}_n$ in $R$ such that $\mathrm{Ann}_{R_0}(y_1\ldots y_s)\subseteq\mathfrak{p}_0$. Finally, a \emph{Krull super-dimension} $\mathrm{Ksdim}(R)$ of $R$ is a couple $n\mid l$, where $n=\mathrm{Kdim}(R_0)$ and $l$ is the supremum of lengths of all systems of odd parameters of $R$. We also use the notations 
$\mathrm{Ksdim}_0(R)$ and $\mathrm{Ksdim}_1(R)$ for $n$ and $l$ respectively. 

In general, the odd super-dimensin $l$ can be infinite. But if $R$ is Noetherian, then $l$ is always finite.
 
The purpose of this work is to define the notion of super-dimension of a supermodule over a super-ring, based on the above notion of Krull super-dimension,  and investigate its properties and several applications. 

The article is organized as follows. In the first section we describe Artinian super-rings. In the second section the notion of super-dimension
of supermodules over certain Noetherian super-rings is introduced and some its elementary properties are studied. We also discuss how the odd parameters relate to odd regular elements in the sense of \cite{sm}. Indeed, one can easily see that any system of odd $M$-regular elements
of a supermodule $M$ form a system of odd parameters of $M$. It was incorrectly stated in Lemma 4.2(3), \cite{zubmas}, that any system of odd regular elements of a super-ring $R$ can be included into a longest system of odd parameters. In the second and third sections we construct
two counterexamples for this statement in the categories of supermodules and super-rings respectively. Moreover, in the category of super-rings we develop a fragment of a \emph{superized Hochschild cohomology} theory, with the help of which the second counterexample is constructed. 

Let $R$ be a (not necessary Noetherian) super-ring and $I$ be a superideal of $R$. We associate with $R$ a graded super-ring
\[\mathsf{gr}_I(R)=\oplus_{k\geq 0} I^k/I^{k+1}\] and a bigraded super-ring 
\[\mathsf{bgr}_I(R)=\oplus_{k, l\geq 0} I_0^k I_1^l R/(I_0^{k+1} I_1^l R +I_0^k I_1^{l+1} R).\] These super-rings are Noetherian, provided $R$ is. If $\mathrm{Ksdim}_0(R)<\infty$, then both $\mathsf{gr}_I(R)$ and $\mathsf{bgr}_I(R)$ satisfy this condition as well. 

Similarly, with any $R$-supermodule $M$ we associate a graded $\mathsf{gr}_I(R)$-supermodule 
\[\mathsf{gr}_I(M)=\oplus_{k\geq 0} I^k M/I^{k+1}M\] and a bigraded $\mathsf{bgr}_I(R)$-supermodule
\[\mathsf{bgr}_I(M)=\oplus_{k, l\geq 0}I_0^k I_1^l M/(I_0^{k+1} I_1^l M +I_0^k I_1^{l+1} M).\] 

Let $R$ be a Noetherian super-ring with $\mathrm{Ksdim}_0(R)<\infty$ and $I$ be a superideal of $R$ that is contained in the radical of $R$. In the fifth section we prove that for any finitely generated $R$-supermodule $M$ there hold $\mathrm{Ksdim}_0(M)=\mathrm{Ksdim}_0(\mathsf{gr}_I(M))$ and
$\mathrm{Ksdim}_1(M)\geq \mathrm{Ksdim}_1(\mathsf{gr}_I(M))$. 

The latter  inequality can be strict in general. However, if $I=I_R=RR_1$, then $\mathrm{Ksdim}(M)=\mathrm{Ksdim}(\mathsf{gr}_I(M))$. In particular, if the elements \[z_1, \ldots , z_t\in I_R/I_R^2\simeq R_1/R_1^3\] form a longest system of odd parameters of the supermodule $\mathsf{gr}_{I_R}(M)$, then their representatives $y_1, \ldots , y_t\in R$ form a longest system of odd parameters of $M$. 

In the six section we develop the dimension theory of certain local super-rings. More precisely, let $R$ be a local super-ring with the maximal superideal $\mathfrak{M}$. Assume that $R$ contains a coefficient field
$K\simeq R_0/\mathfrak{m}$. Let $\mathfrak{N}$ be a $\mathfrak{M}$-primary superideal of $R$ such that $(\mathfrak{n}+R_1^2)/R_1^2$ is generated by the parameters of $\overline{R}=R/I_R\simeq R_0/R_1^2$.  Then $\mathrm{Ksdim}_0(\mathsf{bgr}_{\mathfrak{N}}(R))=\mathrm{Ksdim}_0(R)$ but $\mathrm{Ksdim}_1(\mathsf{bgr}_{\mathfrak{N}}(R))\leq \mathrm{Ksdim}_1(R)$.  However, if $\mathfrak{N}$ has a form $\mathfrak{n}\oplus R_1$, then $\mathrm{Ksdim}(\mathsf{bgr}_{\mathfrak{N}}(R))=\mathrm{Ksdim}(R)$. 

Let $B$ denote $\mathsf{bgr}_{\mathfrak{N}}(R)$ and for each couple of nonnegative integers $k, l$ let $B(k, l)$ denote the corresponding homogeneous component of the bigraded super-ring $B$. We can associate with $B$ a \emph{Hilbert polynomial} $g_{\mathfrak{N}}(x, y)=\sum_{l\geq 0} g_l(x)y^l$ such that 
\[\sum_{0\leq t\leq k}\dim B(t, l) = g_l(k)\]
for all sufficiently large $k$ and for any $l\geq 0$. Then the degree of each polynomial $g_l(x)$ is at most $d=\mathrm{Ksdim}_0(R)$ and if
$\mathfrak{N}_1=R_1$, then
\[\mathrm{Ksdim}_1(R)=\max\{l\mid \ \mbox{the degree of} \ g_l(x) \ \mbox{is equal to} \ d\}.\]

In the fifth and seventh sections we discuss some applications of the above results for the dimension theory of superschemes. For example, we prove that for any flat morphism $f : X\to Y$ of irreducible superschemes of finite type and for any nonsingular point $y\in Y^e$, the odd dimension of each irreducible  component of the fiber $X_y$ is at most $\mathrm{sdim}_1(X)-\mathrm{sdim}_1(Y)$, contrary to the well known fact that its even dimension is equal to $\mathrm{sdim}_0(X)-\mathrm{sdim}_0(Y)$. Moreover, this inequality can be strict in general.   

\section{Artinian super-rings}

As in the classical setting, a super-ring $R$ is called \emph{Artinian} if $R$ satisfies the descending chain condition for super-ideals (DCC).
\begin{lm}\label{first_characterization}
A super-ring $R$ is Artinian if and only if $R_0$ is an Artinian ring and $R_0$-module $R_1$ is Artinian as well.
\end{lm} 
\begin{proof}
Use the same arguments as in Lemma 1.4, \cite{zubmas}.
\end{proof}
The following proposition superizes Proposition (2.C), \cite{mats}.
\begin{pr}\label{characterization_of_Artinian}
A super-ring $R$ is Artinian if and only if it has finite length as $R$-supermodule.
\end{pr}
\begin{proof}
The part "if" is obvious.

Assume that $R$ is Artinian. By Lemma \ref{first_characterization} $R_0$ is Artinian, hence $R_0$ (and $R$ as well) has only finite number of maximal ideals, say $\mathfrak{m}_1, \ldots, \mathfrak{m}_t$ (correspondingly, $\mathfrak{M}_1=\mathfrak{m}_1\oplus R_1, \ldots, \mathfrak{M}_t=\mathfrak{m}_t\oplus R_1$). Moreover, 
$\mathsf{rad}(R_0)=\cap_{1\leq i\leq t}\mathfrak{m}_i=\mathfrak{m}_1\ldots\mathfrak{m}_t$ is nilpotent
 (cf. Proposition (2.C), \cite{mats}). On the other hand, the ideal $R_1^2$ is locally nilpotent, hence $R_1^2\subseteq\mathsf{rad}(R_0)$ and $R_1^2$ is also nilpotent, say $R_1^{2k}=0$ for sufficiently large $k$. 
 
Set $I=\mathfrak{M}_1\ldots\mathfrak{M}_t$. Since $I\subseteq\mathsf{rad}(R_0)+I_R$ and $I_R^{2k}=0$, we have
$I^{2k+l}=0$, provided $\mathsf{rad}(R_0)^l=0$. Set $N=2k+l$. Now, one can construct a chain of super-ideals
\[R\supseteq \mathfrak{M}_1\supseteq\mathfrak{M}_1\mathfrak{M}_2\supseteq\ldots\supseteq\mathfrak{M}_1\ldots\mathfrak{M}_{t-1}\supseteq I\supseteq\] 
\[I\mathfrak{M}_1\supseteq I\mathfrak{M}_1\mathfrak{M}_2\supseteq\ldots\supseteq I\mathfrak{M}_1\ldots\mathfrak{M}_{t-1}\supseteq I^2\supseteq\ldots\supseteq I^N=0,\]
where each quotient of the chain is a vector (super)space over a field $R/\mathfrak{M}_i$ for some $i$, hence a finite dimensional vector (super)space. Proposition is proven.
\end{proof}
\begin{cor}\label{third_characterization}
A super-ring $R$ is Artinian if and only if $R$ is a Noetherian super-ring of Krull super-dimension $0|n$ for some non-negative integer $n$.
\end{cor}
\begin{proof}
If $R$ is Artinian, then Proposition \ref{characterization_of_Artinian} obviously implies that $R$ is Noetherian. Moreover, by Corollary (2.C), \cite{mats}, $\mathrm{Ksdim}_0(R)=\mathrm{Kdim}(R_0)=0$.
Conversely, the same corollary and Lemma 1.4, \cite{zubmas}, infer that if $R$ is Noetherian and $\mathrm{Kdim}(R_0)=0$, then $R_0$ is Artinian and $I_R$ is nilpotent. Repeating the arguments of Proposition \ref{characterization_of_Artinian}, we complete the proof. 
\end{proof}

\section{Super-dimension of supermodules}

Let $B$ be a commutative ring and let $N$ be a $B$-module. Recall that the \emph{dimension} of $N$ is defined by $\dim(N)=\mathrm{Kdim}(B/\mathrm{Ann}_B(N))$.

From now on we assume that $R$ is a Noetherian super-ring with $\mathrm{Kdim}(R_0)<\infty$, unless stated otherwise. For any set of odd elements $y_1, \ldots, y_s\in R_1$ and any subset $L\subseteq\underline{s}=\{1, 2, \ldots, s  \}$, let $y^L$ denote the product
$\prod_{i\in L}y_i$. 

If $M$ is an $R$-supermodule, then $\dim(M)$ is always regarded as the dimension of
$M$ as the $R_0$-module.  
Further, the \emph{super-dimension} of $M$ is defined as
\[\mathrm{sdim}(M)=\mathrm{Ksdim}(R/\mathrm{Ann}_R(M)).\]
More precisely, $\mathrm{sdim}(M)$ is a vector with two coordinates $\mathrm{sdim}_0(M)$ and $\mathrm{sdim}_1(M)$, where $\mathrm{sdim}_0(M)=\dim(M)$. In its turn, $\mathrm{sdim}_1(M)$ is equal to the largest positive integer $l$ such that there are the elements $y_1, \ldots, y_l\in R_1$, which satisfy 
\[(\star) \ \dim(y^{\underline{l}}M))=\dim(M),\]
otherwise $\mathrm{sdim}_1(M)=0$. 

In what follows, any collection of odd elements $y_1, \ldots, y_l$, which satisfies the condition $(\star)$, is called a \emph{system of odd parameters} of $M$. In other words, the elements $y_1, \ldots , y_l\in R_1$ form a system of odd parameters of $M$ if and only if there is a prime chain $\mathfrak{p}_0\subseteq\ldots\subseteq\mathfrak{p}_t$ in $R_0$, such that $\mathrm{Ann}_{R_0}(y^{\underline{l}}M)\subseteq\mathfrak{p}_0$ and 
\[\mathfrak{p}_0/\mathrm{Ann}_{R_0}(M)\subseteq\ldots\subseteq\mathfrak{p}_t/\mathrm{Ann}_{R_0}(M)\]
is a longest prime chain in $R_0/\mathrm{Ann}_{R_0}(M)$. We say that the above system of odd parameters subordinates the prime ideal $\mathfrak{p}_0$.

For example, it is obvious that the Krull super-dimension of a super-ring $R$ coincides with the super-dimension of $R$, regarded as a left $R$-supermodule. Moreover, any system of odd parameters of the super-ring $R$ is the system of odd parameters of the left $R$-supermodule $R$ and vice versa.
\begin{pr}\label{extension_of_prop}
Choose a system of generators of $R_0$-module $R_1$, say $y_1, \ldots, y_d$.
Then the following conditions are equivalent :
\begin{enumerate}
\item There is a system of odd parameters of $M$ of cardinality $l\geq 1$;
\item For some $1\leq i_1<\ldots <i_l\leq d$ the elements $y_{i_1}, \ldots , y_{i_l}$ form a system of odd parameters of $M$;
\item $\dim(R_1^l M)=\dim (M)$.
\end{enumerate}
In particular, we have 
\[\mathrm{sdim}_1(M)=\max\{l\mid \dim(R_1^l M)=\dim (M)\}.\]
\end{pr}
\begin{proof}
It is obvious modification of the proof of Proposition 4.1, \cite{zubmas}.
\end{proof}
\begin{lm}\label{supermodules of finite length}
If $M$ is a finitely generated $R$-supermodule, then the following conditions are equivalent :
\begin{enumerate}
\item $M$ is a supermodule of finite length.
\item $R/\mathrm{Ann}_R(M)$ is an Artinian super-ring.
\item $\mathrm{sdim}(M)=0|n$ for some nonnegative integer $n$.
\item $\dim(M)=0$.
\end{enumerate}
\end{lm}
\begin{proof}
The equivalences $(2)\leftrightarrow (3)$ and $(2)\leftrightarrow (4)$ immediately follow by Corollary \ref{third_characterization}.
If $M$ is an $R$-supermodule of finite length, then any its super-submodule is of finite length. Without loss of a generality, one can assume that $M$ is generated by homogeneous elements $m_1, \ldots , m_t$.
Then $\mathrm{Ann}_R(M)=\cap_{1\leq i\leq t}\mathrm{Ann}_R(m_i)$ and $R/\mathrm{Ann}_R(M)$ is a subdirect product of Artinian super-rings $R/\mathrm{Ann}_R(m_i)$, hence Artinian as well. Thus $(1)$ implies $(2)$. Conversely, there is a natural surjective supermodule morphism $(R/\mathrm{Ann}_R(M))^{\oplus t}\to M$. Therefore, $M$ is an $R$-supermodule of finite length, provided  $R/\mathrm{Ann}_R(M)$ is an Artinian super-ring. Lemma is proven.
\end{proof}
\begin{lm}\label{over_Artinian}
Let $R$ be an Artinian super-ring. Then for any non zero $R$-supermodule $M$ there is
\[\mathrm{sdim}_1(M)=\max\{l\mid R_1^lM\neq 0\}.\]
\end{lm}
\begin{proof}
By Corollary (2.C), \cite{mats}, $R_0$ (respectively, $R$) has finitely many prime ideals (respectively, super-ideals) all of which are maximal. Since $\dim(M)=0$,  a nonempty collection of odd elements
$z_1, \ldots, z_l$ form a system of odd parameters of $M$ if and only if $\mathrm{Ann}_{R_0}(z^{\underline{l}}M)$ is a proper ideal of $R_0$ if and only if $z^{\underline{l}}M\neq 0$. Proposition \ref{extension_of_prop} infers the statement.
\end{proof}
\begin{lm}\label{on even dimension}
We have the following :
\begin{enumerate}
\item If $I$ is an nilpotent superideal of $R$, then $\mathrm{sdim}_0(M/IM)=\mathrm{sdim}_0(M)$.
\item If a super-submodule $N$ of $M$ contains $yM$, where $y$ belongs to a (not necessary longest) system of odd parameters, then $\mathrm{sdim}_0(N)=\mathrm{sdim}_0(M)$.  
\end{enumerate}
\end{lm}
\begin{proof}
$(1)$ Assume that  $I^l=0$ for sufficiently large nonnegative integer $l$. If an even element $r$ satisfies $rM\subseteq IM$, then $r^{l}M=0$, that is $\mathrm{Ann}_{R_0}(M/IM)$ is contained in the radical of $\mathrm{Ann}_{R_0}(M)$. 

$(2)$ If $y=y_1$ and $y_1, \ldots , y_s$ is a system of odd parameters, then
\[\mathrm{Ann}_{R_0}(M)\subseteq \mathrm{Ann}_{R_0}(N)\subseteq \mathrm{Ann}_{R_0}(yM)\subseteq \mathrm{Ann}_{R_0}(y^{\underline{s}}M).\]
\end{proof}
\begin{lm}\label{if even dimensions are the same}
Let $N$ be a super-submodule of $M$, such that $\mathrm{sdim}_0(N)=\mathrm{sdim}_0(M)$. Then
$\mathrm{sdim}_1(N)\leq \mathrm{sdim}_1(M)$. 
\end{lm}
\begin{proof}
Assume that $y_1, \ldots, y_s$ is a system of odd parameters of $N$. Then
\[\mathrm{Ann}_{R_0}(M)\subseteq \mathrm{Ann}_{R_0}(y^{\underline{s}}M)\subseteq \mathrm{Ann}_{R_0}(y^{\underline{s}}N)\]
implies that $y_1, \ldots, y_s$ is a system of odd parameters of $M$. 
\end{proof}
Recall that an odd element $r\in R_1$ is $M$-regular if the kernel of the odd $R$-linear map
$m\mapsto rm, m\in M$, coincides with $rM$ (cf. \cite{sm}, (2.1.1)). A sequence of odd elements $r_1, \ldots, r_t$ is called $M$-regular if each $r_i$ is $M$-regular modulo $(Rr_1+\ldots +Rr_{i-1})M$.
\begin{pr}\label{factoring by parameters}
The following statements hold :
\begin{enumerate}
\item Let $y_1, \ldots, y_t$ be a fragment of a longest system of odd parameters of $M$. Set $I=Ry_1+\ldots Ry_t$. Then 
\[\mathrm{sdim}_0(M)=\mathrm{sdim}_0(IM)=\mathrm{sdim}_0(M/IM)\] and 
\[\mathrm{sdim}_1(M)\geq \mathrm{sdim}_1(IM)\geq \mathrm{sdim}_1(M)-1, \ \mathrm{sdim}_1(M/IM)\geq
\mathrm{sdim}_1(M)-t .\]
\item If $y_1, \ldots, y_t$ is an odd $M$-regular sequence, then 
$\mathrm{sdim}_1(M)\geq t$ and
\[\mathrm{sdim}_1(M/IM)=\mathrm{sdim}_1(y^{\underline{t}}M)\leq \mathrm{sdim}_1(M)-t.\]
Moreover, $\mathrm{sdim}_1(M/IM)$ is equal to $\mathrm{sdim}_1(M)-t$ if and only if $y_1, \ldots, y_t$ can be extended to a longest system of odd parameters of $M$.  
\end{enumerate}
\end{pr}
\begin{proof}
$(1)$ Lemma \ref{on even dimension} implies the first two equations. If $y_1, \ldots , y_l$ is a longest system of odd parameters of $M$, that extends $y_1, \ldots, y_t$, then $y_2, \ldots, y_l$ is obviously
a system of odd parameters of $y_1M$. By Lemma \ref{if even dimensions are the same} we have
\[ \mathrm{sdim}_1(M)\geq \mathrm{sdim}_1(IM)\geq \mathrm{sdim}_1(y_1M)\geq \mathrm{sdim}_1(M)-1. \]
Further, we obviously have
\[\mathrm{Ann}_{R_0}(M)\subseteq\mathrm{Ann}_{R_0}((y^{\underline{l}\setminus\underline{t}}M+IM)/IM)\subseteq \mathrm{Ann}_{R_0}(y^{\underline{l}}M).\]
In other words, $y_{t+1}, \ldots y_l$ form a system of odd parameters for $M/IM$.

$(2)$ The map $m\mapsto y^{\underline{t}}m, m\in M$, induces an isomorphism $M/IM\simeq y^{\underline{t}}M$ of $R$-supermodules, which has parity $|y^{\underline{t}}|$ (see \cite{sm}, Corollary 3.1.2). In particualr, 
$\mathrm{sdim}_1(M/IM)=\mathrm{sdim}_1(y^{\underline{t}}M)$. Moreover, if $ry^{\underline{t}}M=0$, then
$rM\subseteq IM$ and $r^{t+1}M=0$. Therefore, $y_1, \ldots, y_t$ form a system of odd parameters of $M$ and $\mathrm{sdim}_1(M)\geq t$. 

If $y_{t+1}, \ldots, y_{l}$ is a longest system of odd parameters for $M/IM$, then $y_1, \ldots, y_l$ is a system of odd parameters for $M$, hence $l\leq s$. Indeed, one easily sees that
\[\mathrm{Ann}_{R_0}(y^{\underline{l}}M)=\mathrm{Ann}_{R_0}((y^{\underline{l}\setminus\underline{t}}M +IM)/IM).\]
Combining with $(1)$, we obtain that $l=s$ if and only if $y_1, \ldots, y_t$ can be extended to a longest system of odd parameters of $M$.
\end{proof}
\begin{rem}\label{some incorrect prop}
In \cite{zubmas}, Lemma 4.2 (3) erroneously states that any odd regular sequence can be included into the longest system of parameters. In the above proposition we give the correct formulation of this statement.
\end{rem}
Let $A$ be a commutative Artinian ring.
In the rest part of this section we construct a supermodule $M$ over the polynomial super-ring
$R=A[Z_1, Z_2, Z_3, Y]$ with free odd generators $Z_1, Z_2, Z_3, Y$, such that $Y$ is an odd $M$-regular element, but $\mathrm{sdim}_1(M/YM)<\mathrm{sdim}_1(M)-1$. In particular, $Y$ can not be included in any system of odd parameters of $M$.

Let $V$ be an $A$-supermodule of finite length. Set $M=V\oplus\Pi V$. To define a structure of $R$-supermodule on $M$ one needs to define an action of the free generators of $R$ on $M$ by some odd $A$-linear endomorphisms of 
$M$, which supercommute to each other.

Set $Y(v)=\Pi v, Y(\Pi v)=0, v\in V$. Thus it is obvious that $Y$ is $M$-regular. 
\begin{lm}\label{a certain action}
Assume that an action of $R$ on $M$ is defined so that \[ Z_i Z_j M\subseteq YM, 1\leq i, j\leq 3 \ \mbox{and} \
Z_1 Z_2 Z_3 M\neq 0 .\] Then $\mathrm{sdim}_1(M)=3$, but $\mathrm{sdim}_1(M/YM)\leq 1$. 
\end{lm}
\begin{proof}
Since $R$ is Artinian, Lemma \ref{over_Artinian} implies that $\mathrm{sdim}_1(M)=3$. In fact, $R_1^3 M\neq 0$ but $R_1^4M=YZ_1Z_2Z_3M\subseteq YZ_1(YM)=0$. Finally, $Y$ can not be included in any system of
odd parameters of length $3$, since the product of the elements of such system belongs to the superideal
$\sum_{1\leq i< j\leq 3}RYZ_iZ_j$, which acts trivially on $M$. Proposition \ref{factoring by parameters} (2) concludes the proof.
\end{proof}
The action of each $Z_i$ can be expressed as $Z_i(v)=\phi_i(v) +\Pi\psi_i(v)$, where $\phi_i, \psi_i$ are odd and even endomorphisms of $A$-supermodule $V$ respectively. Then
$Z_i(\Pi v)=(Z_i Y)(v)=-Y(Z_i(v))=-Y(\phi_i(v))=-\Pi\phi_i(v)$.
\begin{lm}\label{supercommuting}
The odd endomorphisms of $M$, induced by the free generators of $R$, supercommute to each other and satisfy the conditions of Lemma \ref{a certain action} if and only if the following equations and one inequality hold :
\begin{enumerate}
\item $\phi_i\psi_i-\psi_i\phi_i=0, \phi_i\phi_j=0, 1\leq i, j\leq 3$;
\item $\psi_i\phi_j-\phi_i\psi_j-\phi_j\psi_i+\psi_j\phi_i=0, 1\leq i\neq j\leq 3$;
\item $\phi_1\psi_2\phi_3\neq 0$.
\end{enumerate}   
\end{lm}
\begin{proof}
All these equations can be easily derived from the formulas 
\[(Z_i Z_j)(v)=(\phi_i\phi_j)(v)+\Pi((\psi_i\phi_j-\phi_i\psi_j)(v)), 1\leq i, j\leq 3,\]
and
\[(Z_1 Z_2 Z_3)(v)=\Pi((\phi_1\psi_2\phi_3)(v)).\]
\end{proof}
Consider a (not necessary super-commutative) $A$-superalgebra $B$, generated by the "formal" elements
$\phi'_i, \psi'_i$, subject to the relations (1) and (2) from Lemma \ref{supercommuting}, and the relations $p=0$, where $p$ is any monomial in $\phi'_i, \psi'_i$ of degree at least 4, $1\leq i\leq 3$.
It is clear that $B$ is a graded superalgebra, provided all $\phi'_i$ and $\psi'_i$ are odd and even 
elements of degree 1 respectively. Moreover, $B$ has finite length as an $A$-supermodule.
\begin{lm}\label{(3) holds in B}
The product $\phi'_1\psi'_2\phi'_3$ is not zero in $B$.
\end{lm}
\begin{proof}
One has to show that $\phi'_1\psi'_2\phi'_3$ can not be expressed as any $A$-linear combination of the polynomials $xp$ or $px$, where $p$ is one of the left hand side polynomials in the relations $(1), (2)$ and $x$ is one of the free generators of the absolutely free $A$-superalgebra $A<\phi_i', \psi_i'\mid 1\leq i\leq 3>$. It is obvious that if $p=0$ is a relation (1), then both $xp$ and $px$ do not contain the monomial $\phi'_1\psi'_2\phi'_3$ at all. Therefore, $\phi'_1\psi'_2\phi'_3$ should be an $A$-linear combination of $xp$ and $px$, where $p=0$ is a relation (2). There are only two such polynomials :
\[\phi_1'(\psi_2'\phi_3'-\phi_2'\psi_3'-\phi_3'\psi_2'+\psi_3'\phi_2') \ \mbox{and} \ (\psi'_2\phi'_1-\phi_2'\psi_1'-\phi_1'\psi_2'+\psi_1'\phi_2')\phi_3' .\]
Modulo the relations (1) they congruent to
\[\phi_1'\psi_2'\phi_3'+\phi_1'\psi_3'\phi_2' \ \mbox{and} \ -\phi_2'\psi_1'\phi'_3-\phi_1'\psi_2'\phi_3'.\]
It is clear now that no linear combination of the last polynomials is equal to $\phi_1'\psi_2'\phi_3'$.
Lemma is proven.
\end{proof}
We set $V=B$ and define $\phi_i$ and $\psi_i$ as the operators $b\mapsto \phi_i'b$ and $b\mapsto \psi_i' b, b\in B, 1\leq i\leq 3$.  It is evident that they satisfy (1), (2), (3) of Lemma \ref{supercommuting}.

\section{Odd regular elements and a superized Hochschild cochain complex}

Let $A$ be a super-ring. In this section we describe an homological approach to the following question.
How to construct a super-ring $R$ with an odd regular element $y\in R$ such that $R/Ry\simeq A$?
By the definition, $yR\simeq\Pi A$ as a $\mathbb{Z}$-supermodule, on which $A$ acts as
$a\Pi b=(-1)^{|a|}\Pi(ab), a, b\in A$. 

For the sake of simplicity we assume that the sequence of $\mathbb{Z}$-supermodules \[0\to Ry\to R\to R/Ry\to 0\] is split.
For exampe, this is the case if $R$ is a superalgebra over a field.
Then $R$ is isomorphic to $A\oplus\Pi A$ as a $\mathbb{Z}$-supermodule, and the multiplication on $R$ can be defined as
\[a\circ b=ab+\Pi\pi(a, b), a\circ\Pi b=(-1)^{|a|}\Pi(ab), (\Pi b)\circ a=\Pi(ba), (\Pi a)\circ(\Pi b)=0, a, b\in A,\]
where $\pi : A\otimes A\to A$ is an odd bilinear super-skew symmetric map. The super-skew symmetry means that $\pi(a, a)=0, \pi(b, c)=(-1)^{|b||c|}\pi(c, b), a\in A_1, b, c\in A$. If $2$ is invertible in $A$, then the first condition is redundant. Note also that the regular element $y$ can be identified with 
$\Pi 1$.
\begin{lm}\label{cocycle}
The map $\pi$ defines a super-ring structure on $A\oplus\Pi A$ if and only if for any $a, b, c\in A$ the following conditions hold :
\begin{enumerate}
\item $\pi(1, a)=0$;
\item $\pi(ab, c)-\pi(a, bc)+\pi(a, b)c-(-1)^{|a|}a\pi(b, c)=0$.
\end{enumerate}
\end{lm}
\begin{proof}
Since $1\circ a=a$, the first condition obviously follows. For the second condition just compare
\[(a\circ b)\circ c=(ab)c+\Pi(\pi(ab, c)+\pi(a, b)c)\]
and
\[a\circ(b\circ c)=a(bc)+\Pi(\pi(a, bc)+(-1)^{|a|}a\pi(b, c)).\]
\end{proof}
Assume that $A$ is arbitrary (not necessary super-commutative) super-ring. Let $M$ be  an $A$-superbimodule. 

We define a \emph{superized Hochschild cochain complex} (compare with \cite{weib}, p.301)
\[\tilde{C}^n(A, M)=\mathrm{Hom}_{\mathbb{Z}}(A^{\otimes (n+1)}, M), n\geq 0,\] with coboundary maps
\[\delta_n(f)(a_0, \ldots, a_{n+1})=\sum_{0\leq i\leq n}(-1)^i f(a_0, \ldots, a_i a_{i+1}, \ldots, a_{n+1})\]
\[-(-1)^{|f||a_0|}a_0f(a_1, \ldots, a_{n+1})+(-1)^{n+1}f(a_0, \ldots, a_n)a_{n+1}, f\in \tilde{C}^n(A, M), n\geq 0.\]
Note that each $\delta_n$ is graded with respect to the natural $\mathbb{Z}_2$-grading of $\mathbb{Z}$-supermodule $\tilde{C}^n(A, M)$.
Thus the cohomology groups of this complex, denoted by $SHH^n(A, M)$, are naturally $\mathbb{Z}_2$-graded. 
\begin{lm}\label{cochain complex}
For any $n\geq 0$ we have $\delta_{n+1}\delta_n=0$.
\end{lm}
\begin{proof}
There is
\[\delta_{n+1}(\delta_n(f))(a_0, \ldots, a_{n+2})=\sum_{0\leq i\leq n+1}(-1)^i\delta_n(f)(a_0, \ldots, a_i a_{i+1}, \ldots, a_{n+2})\]
\[-(-1)^{|f||a_0|}a_0\delta_n(f)(a_1, \ldots, a_{n+2})+(-1)^{n+2}\delta_n(f)(a_0, \ldots, a_{n+1})a_{n+2}.\]
In the sum
\[\sum_{0\leq i\leq n+1}(-1)^i\delta_n(f)(a_0, \ldots, a_i a_{i+1}, \ldots, a_{n+2})\]
each summand $f(a_0, \ldots , a_i a_{i+1}, \ldots, a_j a_{j+1}, \ldots, a_{n+2})$ appears two times, once with the sign $(-1)^{i+j-1}$ and once with the sign $(-1)^{i+j}$, hence they are cancelled. Similarly, a summand
$f(a_0, \ldots, a_i a_{i+1}a_{i+2}, \ldots, a_{n+2})$ appears two times, once with the sign $(-1)^{2i}$ and once with the sign $(-1)^{2i+1}$, hence cancelled as well. Thus
\[\sum_{0\leq i\leq n+1}(-1)^i\delta_n(f)(a_0, \ldots, a_i a_{i+1}, \ldots, a_{n+2})=\]\[-(-1)^{|a_0 a_1||f|}a_0 a_1 f(a_2, \ldots, a_{n+2})+(-1)^{n+1}f(a_0a_1, \ldots , a_{n+1})a_{n+2}\]\[-\sum_{0< i  <n+1}(-1)^{i+|a_0||f|}a_0 f(a_1, \ldots, a_i a_{i+1}, \ldots, a_{n+2})\]\[+(-1)^{n+1}\sum_{0< i< n+1}(-1)^i f(a_0, \ldots, a_i a_{i+1}, \ldots, a_{n+1})a_{n+2}\]
\[-(-1)^{n+1+|a_0||f|}a_0f(a_1, \ldots, a_{n+1}a_{n+2})+f(a_0, \ldots, a_n)a_{n+1}a_{n+2}=\]
\[(-1)^{|a_0||f|}a_0[-(-1)^{|a_1||f|}a_1f(a_2, \ldots , a_{n+2})+\sum_{0< i\leq n+1}(-1)^{i+1}f(a_1, \ldots, a_i a_{i+1}, \ldots, a_{n+2})]\]
\[+(-1)^{n+1}[\sum_{0\leq i < n+1}(-1)^i f(a_0, \ldots, a_i a_{i+1}, \ldots, a_{n+1})+(-1)^{n+1}f(a_0, \ldots, a_n)a_{n+1}]a_{n+2}=\]
\[(-1)^{|a_0||f|}a_0[\delta_n(f)(a_1, \ldots, a_{n+2})-(-1)^{n+1}f(a_1, \ldots, a_{n+1})a_{n+2}]\]
\[+(-1)^{n+1}[\delta_n(f)(a_0, \ldots, a_{n+1})+(-1)^{|a_0||f|}f(a_1, \ldots, a_{n+1})a_{n+2}]=\]
\[(-1)^{|a_0||f|}\delta_n(f)(a_1, \ldots, a_{n+2})+(-1)^{n+1}\delta_n(f)(a_0, \ldots, a_{n+1})a_{n+2}.\]
Lemma is proven.
\end{proof}
From now on $A$ is again a super-commutative super-ring and $M$ is a left $A$-supermodule, regarded as a right $A$-supermodule via the right action \[ma=(-1)^{|a||m|}am, a\in A, m\in M .\]

Let $C^n(A, M)$ denote the $\mathbb{Z}$-supersubmodule of $\tilde{C}^n(A, M)$ consisting of all $f\in \tilde{C}^n(A, M)$, such that for any $a_0, \ldots , a_n\in A$ there is
\begin{enumerate}
\item $f(1, a_1, \ldots, a_n)=0$;
\item $f(a_n, a_{n-1}, \ldots,  a_0)=(-1)^{\frac{n(n-1)}{2}+\sum_{0\leq i< j \leq n} |a_i||a_j|}f(a_0, \ldots , a_n)$;
\item if $2$ is not invertible in $A$, then $f(a, \ldots, a)=0$, provided $n\geq 1$ and $a$ is odd. 
\end{enumerate}
\begin{lm}\label{a subcomplex}
If $A$ is a super-commutative super-ring, then for any $n\geq 0$ we have $\delta_n(C^n(A, M))\subseteq C^{n+1}(A, M)$, that is $\{C^n(A, M)\}_{n\geq 0}$ is a subcomplex of the complex $\{\tilde{C}^n(A, M)\}_{n\geq 0}$.
\end{lm}
\begin{proof}
Let $d_n$ and $\epsilon_n$ denote $\sum_{0\leq i< j\leq n}|a_i||a_j|$ and $\frac{n(n-1)}{2}$ respectively. For arbitrary $f\in C^n(A, M)$ we have
\[\delta_n(f)(a_{n+1}, \ldots, a_0)=\sum_{0\leq i\leq n}(-1)^i f(a_{n+1}, \ldots, a_{n+1-i} a_{n-i}, \ldots, a_0)\]
\[-(-1)^{|f||a_{n+1}|}a_{n+1}f(a_n, \ldots, a_0)+(-1)^{n+1}f(a_{n+1}, \ldots, a_1)a_0 =\]
\[(-1)^{\epsilon_n+d_{n+1}+n}\delta_n(f)(a_0, \ldots, a_{n+1}). \]
Similarly, there is
\[\delta_n(f)(1, a_1, \ldots, a_{n+1})=f(a_1, \ldots, a_{n+1})-f(a_1, \ldots, a_{n+1})=0.\]
It is clear that $\delta_n(f)$ atisies the condition $(3)$, whenever $n\geq 1$. If $n=0$, then
\[ \delta_0(f)(a, a)=f(a^2)-(-1)^{|f|}af(a)-f(a)a=0 \]
for $a\in A_1$.
Lemma is proven.
\end{proof}
The cohomolgy groups of the complex $\{C^n(A, M)\}_{n\geq 0}$ are denoted by $SH^n(A, M)$. 

The statement of Lemma \ref{cocycle} can be formulated so that $\pi(a, b)$ defines a super-ring structure on $A\oplus\Pi A$ if and only if $\pi$ is an odd cocycle from $C^1(A, A)$. The resulting super-ring is denoted by
$A_{\pi}$. 

Let $R$ and $R'$ are super-rings with regular elements $y$ and $y'$ respectively. We say that an isomorphism of super-rings $\phi : R\to R'$ is \emph{adapted} to $y$ and $y'$, whenever $\phi(y)=y'$ and
the induced super-ring morphism $R/Ry\to R'/R'y'$ is an isomorphism. If it the case, then $R$ and $R'$ are said to be \emph{adaptively} isomorphic. If we denote $R/Ry\simeq R'/R'y'$ by $A$, then there are cocycles $\pi, \pi' \in C^1(A, A)$ such that $R$ and $R'$ can be identified with $A_{\pi}$ and $A_{\pi'}$
respectively.  
\begin{pr}\label{H^1 and adaptively isomorphic super-rings}
In the above notations, $R$ is adaptively isomorphic to $R'$ if and only if the cosets of $\pi$ and $\pi'$
in $SH^1(A, A)$ are the same. Moreover, 
the exact sequence 
\[0\to Ry\to R\to A\to 0\]
is split if and only if the coset of $\pi$ is zero.
\end{pr}
\begin{proof}
Recall that $R$ and $R'$ are isomorphic to $A\oplus\Pi A$ as $\mathbb{Z}$-supermodules. Let $\phi : R\to R'$ be an adapted isomorphism. It can be defined by 
\[a\mapsto a+\Pi f(a), \Pi a\mapsto \Pi a, a\in A,\]
where $f\in C^0(A)_1$. Then $\phi$ is a super-ring morphism if and only if 
\[\phi(a\circ b)=\phi(ab +\Pi\pi(a, b))=ab +\Pi(f(ab)+\pi(a, b))=
\phi(a)\circ\phi(b)=\]\[ab +\Pi(\pi'(a, b)+f(a)b+(-1)^{|f||a|}af(b))\]
for any $a, b\in A$, hence if and only if $\pi'-\pi=\delta_0(f)$. The proof of the second statement is similar and we leave it for the reader. Proposition is proven.
\end{proof}
In the rest part of this section we construct an Artinian super-ring $R$ with regular element $y\in R$ such that $\mathrm{sdim}_1(R/Ry)<\mathrm{sdim}_1(R)-1$. 

Let $B$ be an Artinian ring. Consider the super-ring $A'=B[Y_1, Y_2, Y_3, Y_4]/I'$, where the super-ideal $I'$ is generated by the elements $Y_{i_1}Y_{i_2}Y_{i_3}, 1\leq i_1<i_2<i_3\leq 4$. It is clear that
$A'$ is an Artinian $B$-superalgebra and $\mathrm{sdim}_1(A')=2$. 

Let $A$ be an $A'$-superalgebra. Choose the elements $t_{i_1 i_2 i_3}\in A_0, 1\leq i_1, i_2, i_3\leq 4$, such that 
$t_{i_{\sigma(1)} i_{\sigma(2)} i_{\sigma(3)}}=(-1)^{\sigma}t_{i_1 i_2 i_3}$ for any substitution $\sigma$ and $t_{i_1 i_2 i_3}=0$ if  one of the indices $i_1, i_2, i_3$ coincides with the other. 
Define a $B$-bilinear map $\pi' : A'\otimes A'\to A$ by 
\begin{enumerate}
\item[(a)] $\pi'(1, A')=\pi'(A', 1)=0$; 
\item[(b)] $\pi'(Y_i, Y_j)=0, 1\leq i, j\leq 4$;
\item[(c)] $\pi'(Y_i Y_j, Y_k)=\pi'(Y_k, Y_i Y_j)=t_{i j k}, 1\leq i, j, k\leq 4$;
\item[(d)] $\pi'(Y_i Y_j, Y_s Y_k)=-t_{s k j}Y_i, 1\leq i\neq j, k\neq s\leq 4$.	
\end{enumerate}	  
\begin{lm}\label{when pi' is a cocylce}
The map $\pi'$ is a cocycle from $C^1(A', A)$ if and only if 
\begin{enumerate}
\item[$(\star)$]  $t_{s k j}Y_i+t_{s k i}Y_j=0, 1\leq i, j, k, s\leq 4$;
\item[$(\star\star)$] $t_{s k i}Y_i=0, 1\leq i, s, k\leq 4$.
\end{enumerate}
\end{lm}
\begin{proof}
The condition $(\star)$ and $(\star\star)$ are derived from (d) and the eccential  equations \[\pi'(Y_i Y_j, Y_s Y_k)=-\pi'(Y_j Y_i, Y_sY_k), \ \pi'(Y_i^2, Y_s Y_k)=0.\]
The super-skew symmetricity of $\pi'$ is equivalent to the condition
\[ t_{i j k}Y_s=t_{s k j}Y_i.\]
Using $(\star)$ and the definition of $t$-s we have
\[t_{i j k}Y_s=t_{j k i}Y_s=-t_{j k s}Y_i=-t_{k s j}Y_i=t_{s k j}Y_i.\]
It sufficies to check the condition (2) from Lemma \ref{cocycle} for the free generators of free $B$-supermodule $A'$ only, that is one can assume that $a, b, c$ belong to \[\{1, Y_i, Y_j Y_k\mid 1\leq i\leq 4, 1\leq j< k\leq 4\}.\] If one of the elements $a, b, c$ is equal to $1$, then (2) holds automatically. Therefore, one can assume that \[a, b, c\in\{Y_i, Y_j Y_k\mid 1\leq i\leq 4, 1\leq j< k\leq 4\}.\]
Note also that if $a, b, c\in\{Y_j Y_k\mid 1\leq j< k\leq 4\}$, then all summands in (2) are equal to zero. 
 
Below we list the rest cases.

If $a=Y_i, b=Y_j, c=Y_k$, then (2) follows from (b) and (c). 

If $a=Y_i, b=Y_j, c=Y_sY_k$, then (2) is equivalent to (d). 

If $a=Y_i, b=Y_s Y_k, c=Y_j$, then (2) is equivalent to $(\star)$.

If $a=Y_s Y_k, b=Y_i, c=Y_j$, then (2) is derived from $(\star)$ and the definition of $t$-s.

If $a=Y_i Y_j , b=Y_s Y_k, c=Y_t$, then (2) is equivalent to the equation
\[(t_{s k j} Y_t+t_{s k t}Y_j)Y_i=0,\] 
which follows from $(\star)$.

If $a=Y_i Y_j, b=Y_t, c=Y_s Y_k$, then (2) is equivalent to the equation
\[t_{i j t}Y_sY_k=t_{s k t}Y_iY_j,\]
that can be derived from $(\star)$ and the definition of $t$-s as
\[t_{i j t}Y_sY_k=t_{j t i}Y_sY_k=-t_{j t s}Y_i Y_k=t_{j t s}Y_kY_i=t_{t s j}Y_kY_i=\]
\[-t_{t s k}Y_j Y_i=t_{t s k}Y_i Y_j=t_{s k t}Y_i Y_j.\]	
Finally, if $a=Y_t, b=Y_i Y_j, c=Y_s Y_k$, then (2) is equivalent to an equation, which is in turn equivalent to the equation from the previous case modulo $(\star)$.
Lemma is proven. 
\end{proof}
Set $B=C[t_{i j k}\mid 1\leq i< j< k\leq 4]/J$, where $C$ is a (commutative) Artinian ring and the ideal $J$ is generated by all products $t_{i j k}t_{s k l}$. Then $B$ is also Artinian. 

For arbitrary triple of indices $i, j, k$ we define $t_{i j k}$ to be zero if $i=j$ or $k\in\{i, j\}$. Otherwise, $i, j, k$ is a permutation of some ordered triple $i'<j'<k'$ and we set $t_{i j k}=(-1)^{\sigma} t_{i' j' k'}$, where $\sigma :
i'\mapsto i, j'\mapsto j, k'\mapsto k$. 

Define $A$ as a quotient of the super-ring $A'$ modulo the superideal $I$ generated by the left-hand sides of $(\star)$ and $(\star\star)$. 
The following lemma is obvious.
\begin{lm}
We have $\pi'(I, A')=0$, hence $\pi'$ naturally induces a cocylce $\pi\in C^1(A, A)$.	
\end{lm}
\begin{lm}
The element $t_{1 2 3}Y_4$ does not belong to $I$. 
\end{lm}
\begin{proof}
If this element does, then it should be a $C$-linear combination of the left-hand sides of $(\star)$ and $(\star\star)$. 
The $C$-module $A'_1$ is freely generated by the elements 
\[t_{ijk}Y_s, 1\leq i< j< k\leq 4, s\in\{i, j, k\},\]
and
\[v_1=t_{234}Y_1, v_2=t_{134}Y_2, v_3=t_{234}Y_3, v_4=t_{123}Y_4.\]
The first elements are just left-hand sides of $(\star)$. 
Let $M$ and $N$ denote the $C$-span of the left-hand sides of $(\star\star)$ and the $C$-span of $v_1, v_2, v_3, v_4$ respectively. It is clear that
$A'_1=M\oplus N$.

If an element $z=t_{s k j}Y_i+t_{s k i} Y_j$ does not belong to $M$, then 
the indices $i, j, s, k$ are pairwise different and $z\in N$. Without loss of a generality one can assume that $s< k$ and $i< j$. There are six such elements :
\[z_1=t_{1 4 3}Y_2+t_{1 4 2}Y_3=-t_{1 3 4}Y_2-t_{1 2 4}Y_3=-v_2-v_3,\]
\[z_2=t_{1 3 4}Y_2+t_{1 3 2 }Y_4=t_{1 3 4}Y_2-t_{1 2 3}Y_4=v_2-v_4,\]	
\[z_3=t_{1 2 4}Y_3+t_{1 2 3}Y_4=v_3+v_4,\]	
\[z_4=t_{234}Y_1+t_{231}Y_4=t_{234}Y_1+t_{123}Y_4=v_1+v_4,\]
\[z_5=t_{243}Y_1+t_{241}Y_3=-t_{234}Y_1+t_{124}Y_3=-v_1+v_3,\]
\[z_6=t_{342}Y_1+t_{341}Y_2=t_{234}Y_1+t_{134}Y_2=v_1+v_2.\]
Straightforward calculations show that 
\[z_1=-z_5-z_6, z_2=-z_4+z_6, z_3=z_4+z_5\]
and the $C$-submodule $N'=\sum_{1\leq i\leq 6} Cz_i$ is freely generated by the elements $z_4, z_5, z_6$. Moreover, $v_4$ does not belong to
$N'$. Lemma is proven.
\end{proof}	
Set $R=A_{\pi}$. Recall that $R$ has an regular element $y=\Pi 1$. Then the odd super-dimension of $A\simeq R/Ry$ is at most $2$. On the other hand, $Y_1\circ Y_2\circ Y_3\circ Y_4=\Pi(\pi(Y_1 Y_2, Y_3)Y_4)=\Pi(t_{12, 3}Y_4)\neq 0$, that is $\mathrm{sdim}_1(R)=4$. 
	
\section{Graded and bigraded super-rings and supermodules associated with Noetherian super-rings and their supermodules}

Recall that a (not necessary Noetherian) super-ring $R$ is called \emph{graded} if $R=\oplus_{n\geq 0} R(n)$, where $R(n)$ is an $R(0)$-super-submodule of $R$ and $R(n)R(n')\subseteq R(n+n')$ for any nonegative integers $n, n'$.
Respectively, an $R$-supermodule $M$ is called \emph{graded} if $M=\oplus_{n\geq 0}M(n)$, where $M(n)$
is an $R(0)$-super-submodule of $M$ and $R(n)M(n')\subseteq M(n+n')$ for any nonegative integers $n, n'$.

A super-ring $R$ is called \emph{bigraded} if 
$R=\oplus_{k, l\geq 0}R(k, l),$ and for any $k, k', l, l'\geq 0$ the following conditions hold :
\begin{enumerate}
\item $R(k, l)$ is a $R(0, 0)$-super-submodule of $R$.
\item $R(k, l)R(k', l')\subseteq R(k+k', l+l')$. 
\end{enumerate}
Respectively, an $R$-supermodule $M$ is called \emph{bigraded} if $M=\oplus_{k, l\geq 0}M(k, l),$ and for any $k, k', l. l'\geq 0$ the following conditions hold :
\begin{enumerate}
\item $M(k, l)$ is a $R(0, 0)$-super-submodule of $M$.
\item $R(k, l)M(k', l')\subseteq M(k+k', l+l')$.
\end{enumerate}

It is clear that a bigraded super-ring $R$ is also a graded super-ring with respect to
the grading $R(n)=\oplus_{k+l=n} R(k, l), n\geq 0$. Moreover, if $M$ is a bigraded $R$-supermodule, then
it is also graded $R$-supermodule with respect to the similar grading $M(n)=\oplus_{k+l=n}M(k, l), n\geq 0$.

The polynomial superalgebra $A[X_1, \ldots, X_r\mid Y_1, \ldots, Y_s]$ is a bigraded superalgebra, where each "scalar" $a\in A$ has degree $(0, 0)$, $X_i$ and $Y_j$ have degrees $(1, 0)$ and $(0, 1)$ respectively, $1\leq i\leq r, 1\leq j\leq s$.

Let $I$ be a super-ideal of $R$. Then we associate with $R$ a graded super-ring $\mathsf{gr}_I(R)$ 
such that $\mathsf{gr}_I(R)(n)=I^n/I^{n+1}, n\geq 0$. Moreover, if $M$ is a {left} $R$-supermodule, then one can define a graded $\mathsf{gr}_I(R)$-supermodule $\mathsf{gr}_I(M)$ such that $\mathsf{gr}_I(M)(n)=I^n M/I^{n+1} M$. 

Let $I_{k, l}$ denote the super-ideal $I_0^k I_1^lR$. It is obvious that $R=I_{0, 0}$ and $I=I_{1, 0}+I_{0, 1}$.
Furthermore, if $M$ is an $R$-supermodule, then we set $M_{k, l}=I_{k, l}M$. Since $I_{k, l}I_{k', l'}\subseteq
I_{k+k', l+l'}$, the superspaces 
$\oplus_{k, l\geq 0}I_{k, l}/I I_{k, l}$ 
and $\oplus_{k, l\geq 0}M_{k, l}/IM_{k, l}$
are a bigraded superalgebra and a bigraded supermodule over it, respectively. We denote them by
$\mathsf{bgr}_{I}(R)$ and $\mathsf{bgr}_{I}(M)$ correspondingly. 

Assume again that $R$ is a Noetherian super-ring with $\mathrm{Ksdim}_0(R)=\mathrm{Kdim}(R_0)<\infty$. 
\begin{lm}\label{graded is still Noetherian}
The super-rings $\mathsf{gr}_{I}(R)$ and $\mathsf{bgr}_{I}(R)$ are Noetherian for any super-ideal $I$ of $R$. Moreover, $\mathrm{Ksdim}_0(\mathsf{gr}_I(R))<\infty$ and $\mathrm{Ksdim}_0(\mathsf{bgr}_{I}(R))<\infty$ as well.
\end{lm}
\begin{proof}
There is a natural superalgebra epimorphism $\mathsf{bgr}_{I}(R)\to \mathsf{gr}_I(R)$ that is induced by the embeddings $I_{k, l}\to I^{k+l}$. Thus all we need is to prove that $\mathsf{bgr}_{I}(R)$ is a Noetherian super-ring and its even part has finite Krull dimension.

Assume that $I$ is generated by $r$ even and $l$ odd elements $x_1, \ldots , x_r$ and $y_1, \ldots, y_l$ respectively. There is a natural epimorphism \[(R/I)[X_1, \ldots, X_r\mid Y_1, \ldots , Y_s]\to\mathsf{bgr}_{I}(R),\] of bigraded superalgebras, 
induced by the map 
\[X_i\mapsto x_i +I I_{1, 0}, \ Y_j\mapsto y_j+I I_{0, 1}, 1\leq i\leq r, 1\leq j\leq s.\] Therefore, one has to show that $(R/I)[X_1, \ldots, X_r\mid Y_1, \ldots , Y_s]$ is a Noetherian super-ring and its even part has finite Krull dimension. 

Note that $(R/I)[X_1, \ldots, X_r\mid Y_1, \ldots , Y_s]$ is a finitely generated 
$(R/I)_0[X_1, \ldots, X_r]$-module. Since $(R/I)_0[X_1, \ldots, X_r]$ is a Noetherian ring, the first statement follows.

Finally, $((R/I)[X_1, \ldots, X_r\mid Y_1, \ldots , Y_s])_0$ is integral over $(R/I)_0[X_1, \ldots, X_r]$, and
Theorem 20 and Theorem 22 from \cite{mats} imply the second statement. 
\end{proof}

\section{Grading and super-dimension of supemodules}
Let $R$ be a Noetherian super-ring with $\mathrm{Kdim}(R_0)<\infty$. Let $I$ be a superideal of $R$ such that $I\subseteq\mathsf{rad}(R)$. The graded super-ring $\mathsf{gr}_{I}(R)$ is denoted by $B$. Then $B=\oplus_{n\geq 0} B(n)$, where $B(n)=I^n/I^{n+1}, n\geq 0$. 
Observe that for any positive integer $n$, there holds $B(n)=B(1)^n$.

Let $M$ be a finitely generated $R$-supermodule. Let $C$ denote $\mathsf{gr}_{I}(M)$. Then $C=\oplus_{n\geq 0} C(n)$, where 
$C(n)=I^n M/I^{n+1}M, n\geq 0$.
As above, for any positive integer $n$ there is $C(n)=B(1)^n C(0)$. In particular, a homogeneous element 
$b\in B(k)$ belongs to $\mathrm{Ann}_{B}(C)$ if and only if $b C(0)=0$. Furthermore, since the super-ideal
$\oplus_{n\geq 1} B(n)$ is nilpotent, we have
\[\mathrm{sdim}_0(C)=\mathrm{Kdim}(B(0)_0/\mathrm{Ann}_{B(0)_0}(C(0))).\]

The prime (super)spectrum of $R$ can be naturally identified with the prime (super)spectrum of $B$ via
\[\mathfrak{P}\mapsto \widetilde{\mathfrak{P}}=\mathfrak{P}/I\oplus I/I^2\oplus\ldots,\]
so that 
\[\mathsf{gr}_{I_{\mathfrak{P}}}(M_{\mathfrak{P}})\simeq \mathsf{gr}_{I}(M)_{\widetilde{\mathfrak{P}}}=B_{\widetilde{\mathfrak{P}}}.\]
\begin{lm}\label{weaker_inequality_II}
The following  hold 
\[\mathrm{sdim}_0(M)=\mathrm{sdim}_0(\mathsf{gr}_{I}(M)) \ \mbox{and} \ \mathrm{sdim}_1(M)\geq \mathrm{sdim}_1(\mathsf{gr}_{I}(M)).\]
\end{lm}
\begin{proof}
By the above remark $\mathrm{sdim}_0(\mathsf{gr}_{I}(M))=\mathrm{sdim}_0(M/IM)$ and Lemma \ref{on even dimension} infers the first equality.
Further, if the elements $z_1, \ldots, z_t$ generate $B(0)_1\oplus B(1)_1$ as a $B(0)_0$-module, then
they generate $B_1$ as a $B_0$-module. Moreover, their representatives $y_1, \ldots , y_t\in R_1$ generate
$R_1$ as an $R_0$-module (cf. \cite{mats}, Corollary (1.M)(b)). By Proposition \ref{extension_of_prop},
for some $1\leq s_1<\ldots<s_l\leq t$ the elements $z_{s_1}, \ldots, z_{s_l}$ form a longest system of odd parameters of $C=\mathsf{gr}_{I}(M)$. In other words, there is a prime ideal $\mathfrak{q}$ of $R_0$ such that the system $z_{s_1}, \ldots, z_{s_l}$ subordinates the prime ideal $\mathfrak{q}/I_0\oplus (\oplus_{n\geq 1}B(n)_0)$, that is $\mathrm{Ann}_{B(0)_0}(z^S C)\subseteq\mathfrak{q}/I_0$, where
$S=\{s_1, \ldots, s_l\}$ and $\mathfrak{q}/I_0$ is congruent to the first member of a longest prime chain in $B(0)_0/\mathrm{Ann}_{B(0)_0}(C(0))$.
We have
\[(\mathrm{Ann}_{R_0}(M) +I_0)/I_0\subseteq \mathrm{Ann}_{B(0)_0}(C(0))\subseteq \mathrm{Ann}_{B(0)_0}(z^S C).\]
Moreover, Lemma \ref{on even dimension} implies that  $\mathrm{Ann}_{B(0)_0}(C(0))$ is contained in the radical of the ideal
$(\mathrm{Ann}_{R_0}(M) +I_0)/I_0$, hence $\mathfrak{q}/\mathrm{Ann}_{R_0}(M)$ is the first member of a longest prime chain in $R_0/\mathrm{Ann}_{R_0}(M)$. Since 
\[(\mathrm{Ann}_{R_0}(y^S M)+I_0)/I_0\subseteq \mathrm{Ann}_{B(0)_0}(z^S C),\] we conclude that
$\mathrm{Ann}_{R_0}(y^S M)\subseteq\mathfrak{q}$. Lemma is proven.
\end{proof}
\begin{pr}\label{useful_inequality_II}
Let $\mathfrak{p}$ be a prime ideal of $R_0$, such that $\mathfrak{p}/\mathrm{Ann}_{R_0}(M)$ is the first member of a longest prime chain of $R_0/\mathrm{Ann}_{R_0}(M)$.
Set $\mathfrak{P}=\mathfrak{p}\oplus R_1$. Then
\[\mathrm{sdim}_1(M_{\mathfrak{P}})=\max\{k\mid (R_1^kM)_{\mathfrak{P}}\neq 0\}\leq\mathrm{sdim}_1(M).\]
If there is a longest system of odd parameters in $M$, which subordinates  $\mathfrak{p}$, then the above inequality turns into equation.
\end{pr}
\begin{proof}
We use the following elementary observations. For any $z_1, \ldots, z_t\in R_1$ and $s_1, \ldots , s_t\in S=R_0\setminus\mathfrak{p}$, there holds
\[\mathrm{Ann}_{S^{-1}R_0}(\frac{z_1}{s_1}\ldots \frac{z_t}{s_t} M)=S^{-1}\mathrm{Ann}_{R_0}(z^{\underline{t}} M).\]
Besides, $\mathrm{Ann}_{R_{\mathfrak{P}}}(M_{\mathfrak{P}})=S^{-1}\mathrm{Ann}_{R}(M)$ and $R_{\mathfrak{P}}/\mathrm{Ann}_{R_{\mathfrak{P}}}(M_{\mathfrak{P}})$ is an Artinian super-ring with the unique prime
superideal $\mathfrak{P}_{\mathfrak{P}}/\mathrm{Ann}_{R_{\mathfrak{P}}}(M_{\mathfrak{P}})=S^{-1}(\mathfrak{P}/\mathrm{Ann}_R(M))$.
Since the super-dimension of $M_{\mathfrak{P}}$ as an $R_{\mathfrak{P}}$-supermodule is the same as the super-dimension of $M_{\mathfrak{P}}$ as an $R_{\mathfrak{P}}/\mathrm{Ann}_{R_{\mathfrak{P}}}(M_{\mathfrak{P}})$-supermodule, Lemma \ref{over_Artinian} implies the first equality.
	
Further, 
the elements $\frac{z_1}{s_1}, \ldots , \frac{z_t}{s_t}$ form a system of odd parameters  if and only if
$\mathrm{Ann}_{S^{-1}R_0}(\frac{z_1}{s_1}\ldots \frac{z_t}{s_t} M)\subseteq S^{-1}\mathfrak{p}$ if and only if
$\mathrm{Ann}_{R_0}(z^{\underline{t}}M)\subseteq\mathfrak{p}$. Proposition is proven.  
\end{proof}
Let $P(M)$ denote the subset of the prime spectrum of $R$ consisting of all $\mathfrak{P}$ such that $\mathrm{Ann}_R(M)\subseteq\mathfrak{P}$ and $\mathfrak{p}/\mathrm{Ann}_{R_0}(M)$ is the first member of a longest prime chain in $R_0/\mathrm{Ann}_{R_0}(M)$. The following corollary is now evident.
\begin{cor}\label{another def of odd dim}
There is
\[\mathrm{sdim}_1(M)=\max_{\mathfrak{P}\in P(M)}\max\{k\mid (R_1^kM)_{\mathfrak{P}}\neq 0\}. \]
\end{cor}
In general, the inequality in Lemma \ref{weaker_inequality_II} is not necessary equation. But in the following important case it is.
\begin{theorem}\label{grading_and_dimension_II}
If $I=I_R$, then 
\[\mathrm{sdim}(M)=\mathrm{sdim}(\mathsf{gr}_{I}(M)).\]
\end{theorem}
\begin{proof}
Let $y_1, \ldots , y_t$ be the same elements as in Lemma \ref{weaker_inequality_II}, and assume that
for some $1\leq i_1<\ldots<i_l\leq t$ the elements $y_{i_1}, \ldots, y_{i_l}$ form a longest system of odd parameters of $M$. Let $\mathfrak{q}$ be a prime ideal of $R_0$ to which this system subordinates. Set $\mathfrak{Q}=\mathfrak{q}\oplus R_1$. 

All we need is to show that $R_{\mathfrak{Q}}$-supermodule $M_{\mathfrak{Q}}$ satisfies the conclusion of our theorem. In fact, if it is the case, then 
\[\mathrm{sdim}_1(M)=\mathrm{sdim}_1(M_{\mathfrak{Q}})=\mathrm{sdim}_1(\mathsf{gr}_{I_{\mathfrak{Q}}}(M_{\mathfrak{Q}}))=\mathrm{sdim}_1(\mathsf{gr}_{I}(M)_{\widetilde{\mathfrak{Q}}})\]\[\leq\mathrm{sdim}_1(\mathsf{gr}_{I}(M)),\]
and by Lemma \ref{weaker_inequality_II} $M$ satisfies the conclusion of our theorem as well. 

Therefore, without loss of a generality, one can assume that $R$ is a local super-ring with the maximal superideal $\mathfrak{Q}$ such that $\mathfrak{Q}/\mathrm{Ann}_R(M)$ is the unique prime superideal of the Artinian super-ring $R/\mathrm{Ann}_R(M)$. Arguing as in Proposition \ref{useful_inequality_II}, we obtain $\mathrm{sdim}_1(M)=\max\{l\mid R_1^l M\neq 0\}$. On the other hand, $B=\mathsf{gr}_I(R)$ is a local super-ring with the maximal superideal $\widetilde{\mathfrak{Q}}$ and the arguments of Lemma \ref{weaker_inequality_II} show that $\widetilde{\mathfrak{Q}}/\mathrm{Ann}_B(C)$ is the unique prime superideal of the Artinian super-ring
$B/\mathrm{Ann}_B(C)$ as well. In particular,  $\mathrm{sdim}_1(C)=\max\{l\mid B_1^l C\neq 0\}$.
Then we have 
\[\max\{l\mid (B_1)^l C\neq 0\}=\max\{l\mid B(1)^l C\neq 0\}=\max\{l\mid C(l)\neq 0\}.\]
On the other hand, $(R_1)^l M\neq 0$ if and only if $(R_1)^l M\neq (R_1)^{l+1} M$ if and only if $C(l)\neq 0$, hence our theorem follows. 
\end{proof}
\begin{cor}\label{corollary for super-rings}
For any super-ring $R$ there holds $\mathrm{Ksdim}(R)=\mathrm{Ksdim}(\mathsf{gr}_{I_R}(R))$.
\end{cor}
\begin{cor}
If the elements $z_1, \ldots , z_t\in I_R/I_R^2\simeq R_1/R_1^3$ form a longest system of odd parameters of the supermodule
$\mathsf{gr}_{I_R}(M)$, then their representatives $y_1, \ldots , y_t$ form a longest system of odd parameters of $M$.	
\end{cor}
\begin{example}
Let $M$ be the $A[Z_1, Z_2, Z_3, Y]$-supermodule from the final part of the third section. Set $I=RY$. Then $R\simeq\mathsf{gr}_I(R)$ and $\mathrm{sdim}_1(\mathsf{gr}_I(M))=2 <\mathrm{sdim}_1(M)$. Indeed, $M/YM\simeq V$ and $YM\simeq\Pi V$ as $A$-supermodules. Moreover, each $Z_i$ acts on them by the rule $Z_i(v)=\phi_i(v), Z_i(\Pi v)=-\Pi(\phi_i(v)), v\in V$. In particular, any product $Z_i Z_j$ acts trivially on $\mathsf{gr}_I(M)$, hence $Z_1, Y$ is one of the longest systems of odd parameters   
of $\mathsf{gr}_I(M)$.
\end{example}

Let $X$ be an irreducible superscheme of finite type over a field $K$. Recall that the \emph{super-dimension} of 
$X$ is defined as $\mathrm{Ksdim}(\mathcal{O}_X(U))$, where $U$ is arbitrary (nonempty) open affine super-subscheme of $X$ (cf. \cite{zubmas}). 

Let $\mathcal{J}$ be a locally nilpotent quasi-coherent superideal sheaf on $X$. Then one can associate with $X$ an 
irreducible superscheme $\mathsf{gr}_{\mathcal{J}}(X)$ with the same \emph{underlying topological space} $X^e$ and with the structural superalgeba sheaf to be the sheafification of the presheaf $\oplus_{n\geq 0}\mathcal{J}^n/\mathcal{J}^{n+1}$.   
The functor $X\to \mathsf{gr}_{\mathcal{J}}(X)$ is an endofunctor of the category of superschemes of finite type that takes immersions to immersions. Moreover, if $\mathcal{J}=\mathcal{I}_X=(\mathcal{O}_X(\mathcal{O}_X)_1)^+$, a superscheme morphism $f : X\to Y$ is an isomorphism if and only if $\mathsf{gr}_{\mathcal{I}_X}(f)$ is (see \cite{maszub2}, Proposition 9.3). 
\begin{pr}
For any locally nilpotent quasi-coherent superideal sheaf $\mathcal{J}$ on $X$ there holds 
\[\mathrm{sdim}_0(X)=\mathrm{sdim}_0(\mathsf{gr}_{\mathcal{J}}(X)), \ \mathrm{sdim}_1(X)\geq\mathrm{sdim}_1(\mathsf{gr}_{\mathcal{J}}(X)),\]
and 
\[\mathrm{sdim}(X)=\mathrm{sdim}(\mathsf{gr}_{\mathcal{I}_X}(X)).\]
\end{pr}
\begin{proof}
Arguing as in the proof of Proposition 9.3 (1) , \cite{maszub2}, one can assume that $X$ is affine, say $X\simeq\mathrm{SSpec}(A)$. Then Corollary \ref{corollary for super-rings} concludes the proof.	
\end{proof}

\section{Dimension theory of certain local superalgebras}

Let $A$ be a Noetherian super-ring, $I$ be a superideal of $A$ and $M$ be an $A$-supermodule. The $I$-adic completions $\varprojlim_n A/I^n $ and $\varprojlim_n M/I^n M$ of $A$ and $M$ are denoted by $\widehat{A}$ and $\widehat{M}$ respectively.
It is evident that $\widehat{M}$ has a natural structure of $\widehat{A}$-supermodule.
\begin{lm}\label{completion as a functor}
The following statements hold :
\begin{enumerate}
\item If $M$ is a finitely generated $A$-supermodule, then $\widehat{M}$ is canonically isomorphic to $M\otimes_A\widehat{A}$.
\item $\widehat{A}$ is a flat $A$-supermodule. 
\end{enumerate}	
\end{lm}
\begin{proof}
Use Proposition 1.9, \cite{zubmas}, and copy the proof of Theorem 55 from \cite{mats}.		
\end{proof}

Let $R$ be a local Noetherian super-ring with the maximal superideal $\mathfrak{M}=\mathfrak{m}\oplus R_1$. Let $\mathfrak{N}$ be a (not necessary prime) $\mathfrak{M}$-\emph{primary} superideal, that is 
$\mathfrak{N}_0=\mathfrak{n}$ is a $\mathfrak{m}$-primary ideal. 
\begin{lm}\label{characterization of primary}
A superideal $\mathfrak{N}$ is $\mathfrak{M}$-primary if and only if $R/\mathfrak{N}$ is Artinian if and only
if $\mathfrak{M}^t\subseteq\mathfrak{N}$ for some positive integer $t$.	
\end{lm}
\begin{proof}
Use Corollary \ref{third_characterization} and note that by Lemma 1.5, \cite{zubmas}, the $\mathfrak{M}$-adic topology on $R$ coincides with the $\mathfrak{m}$-adic topology as well as the $\mathfrak{N}$-adic topology on $R$ coincides with the $\mathfrak{n}$-adic topology. 	
\end{proof}
From now on we assume that $R$ contains a field, unless stated otherwise.
\begin{lm}\label{integral over nilpotents}
Let $A$ be a Noetherian ring and I be a nilpotent ideal of $A$. If $B$ is a subring of $A$ such that $A/I$ is a finitely generated $(B+I)/I$-module, then $A$ is a finitely generated $B$-module also.
\end{lm}
\begin{proof}
Just note that each quotient $I^k/I^{k+1}$ is a finitely generated $B/I$-module, hence a finitely generated $B$-module as well. 	
\end{proof}	
\begin{lm}\label{even dim for bigraded}
Assume that the residue classes of the elements $x_1, \ldots, x_d$ form a system of (even) parameters of $\overline{R}=R_0/R_1^2$. If $\mathfrak{n}$ is generated by these elements modulo $R_1^2$, then
$d=\mathrm{Ksdim}_0(R)=\mathrm{Ksdim}_0(\mathsf{bgr}_{\mathfrak{N}}(R))$. 
\end{lm}
\begin{proof}
Let $B$ denote $\mathsf{bgr}_{\mathfrak{N}}(R)$. Then 
\[B=\oplus_{k, l\geq 0} B(k, l), B(k, l)=\mathfrak{n}^k\mathfrak{N}^l_1R/(\mathfrak{n}^{k+1}\mathfrak{N}^l_1R+\mathfrak{n}^k\mathfrak{N}^{l+1}_1R), k, l\geq 0.\]
Observe that $B(0, 0)_0\simeq R_0/\mathfrak{n}$ is a local algebra with the nilpotent maximal ideal $\mathfrak{m}/\mathfrak{n}$, hence complete. By Cohen structure theorem
$B$ contains a \emph{coefficient field} $K\simeq R/\mathfrak{M}$ and $B(0, 0)=R/\mathfrak{N}$ is a finite dimensional $K$-superalgebra.

Since the superideal $\oplus_{k\geq 0, l> 0} B(k, l)$ is nilpotent, we have $\mathrm{Ksdim}_0(B)=\mathrm{Ksdim}_0(C)$, where $C$ is the (graded) super-subring $\oplus_{k\geq 0} B(k, 0)$. Moreover, the ring
\[C_0=\oplus_{k\geq  0}\mathfrak{n}^k/(\mathfrak{n}^{k+1}+\mathfrak{n}^k\mathfrak{N}_1R_1)\]
is a quotient of $\mathsf{gr}_{\mathfrak{n}}(R_0)$ modulo the nilpotent ideal
\[\oplus_{k\geq  0}(\mathfrak{n}^k\mathfrak{N}_1R_1 +\mathfrak{n}^{k+1})/\mathfrak{n}^{k+1} ,\]
hence $\mathrm{Ksdim}_0(B)=\mathrm{Kdim}(\mathsf{gr}_{\mathfrak{n}}(R_0))$. 

Let $\overline{\mathfrak{n}}$ denote $(\mathfrak{n}+R_1^2)/R_1^2$. The ring $\mathsf{gr}_{\overline{\mathfrak{n}}}(\overline{R})$ is a quotient 
of $\mathsf{gr}_{\mathfrak{n}}(R_0)$ modulo the nilpotent ideal 
\[\oplus_{k\geq 0}(\mathfrak{n}^k\cap R_1^2 +\mathfrak{n}^{k+1})/\mathfrak{n}^{k+1},\]
that implies $\mathrm{Ksdim}_0(B)=\mathrm{Kdim}_0(\mathsf{gr}_{\overline{\mathfrak{n}}}(\overline{R}))$.
Arguing as in Proposition 11.20 and Corollary 11.21, \cite{atmac}, one can show that $\mathsf{gr}_{\overline{\mathfrak{n}}}(\overline{R})$ is a finitely generated module over the polynomial subalgebra $K[t_1, \ldots, t_d]$,
where $t_i$ is $x_i+R_1^2 \pmod{\overline{\mathfrak{n}}^2}, 1\leq i\leq d$. Theorem 20 from \cite{mats} concludes the proof.	
\end{proof}	
Assume that $R$ is complete with respect to the $\mathfrak{M}$-adic or, equivalently, to the $\mathfrak{m}R$-adic topology.
The ring $R_0$ is also complete and therefore, it contains a coefficient field $K\simeq R/\mathfrak{M}$.
Then the map $X_i\mapsto z_i=x_i+R_1^2, 1\leq i\leq d$, is extended for a continuous monomorphism $K[[X_1, \ldots, X_d]]\to \overline{R}$ of complete local algebras. Besides, $\overline{R}$ is a finitely generated module over
$K[[z_1, \ldots, z_d]]$ (see \cite{atmac}, Chapter 11, Exercise 2). Lemma \ref{integral over nilpotents} implies that $R_0$ is a finitely generated $K[[x_1, \ldots, x_d]]$-module. Moreover, $K[[x_1, \ldots, x_d]]$ is a power series algebra,  freely generated by the elements $x_1, \ldots, x_d$ as a \emph{topological} algebra. Let $D$ denote $K[[x_1, \ldots, x_d]]$. 
\begin{lm}\label{Noether normalization for complete local}
The following statements are equivalent :
\begin{enumerate}
\item The odd elements $y_1, \ldots , y_l$ form a system of odd parameters of $R$;
\item $\mathrm{Ann}_D(y^{\underline{l}})=0$;
\item The map $Y_k\mapsto y_k, 1\leq k\leq l$, is extended for the natural isomorphism 
\[K[[X_1, \ldots, X_d\mid Y_1, \ldots, Y_l]]\simeq D[y_1, \ldots, y_l].\] 
\end{enumerate}	
\end{lm}
\begin{proof}
The proof can be copied from Proposition 4.3, \cite{zubmas}.	
	\end{proof}
Let $y_1, \ldots, y_s$ generate $R_1$ as a $D$-module. The proof of the following lemma is an obvious modification of the proof of Lemma 4.4, \cite{zubmas}.
\begin{lm}\label{power series analog}
There holds
\[\mathrm{Ksdim}_1(R)=\max\{|I|\mid I\subseteq\underline{s} \ \mbox{and} \ \mathrm{Ann}_D(y^I)=0 \}=\max\{l  \mid \mathrm{Ann}_D(R_1^l)=0 \}.\]
\end{lm}
\begin{pr}\label{upper bound for odd}
Under the conditions of Lemma \ref{even dim for bigraded} we have \[\mathrm{Ksdim}_1(\mathsf{bgr}_{\mathfrak{N}}(R))\leq\mathrm{Ksdim}_1(R).\]
\end{pr}
\begin
{proof}
Let $\widehat{R}$ denote the completion of $R$ in the $\mathfrak{N}$-adic topology. Since this topology coincides with the $\mathfrak{M}$-adic 
topology, Theorem 5.12 from \cite{zubmas} implies $\mathrm{Ksdim}(R)=\mathrm{Ksdim}(\widehat{R})$. Moreover, Lemma \ref{completion as a functor} implies that $\mathsf{bgr}_{\mathfrak{N}}(R)$ is naturally isomorphic to $\mathsf{bgr}_{\widehat{\mathfrak{N}}}(\widehat{R})$. Therefore, without loss of a generality one can assume that $R$ is complete. 

Using Lemma \ref{integral over nilpotents} and the elementary observation that some elements are algebraically independent (over the ground field $K$)  whenever they are algebraically independent modulo an ideal, we obtain that the elements   
$\overline{x_j}= x_j \pmod{\mathfrak{N}_{1, 0}\mathfrak{N}}$ freely generate a polynomial subalgebra $K[\overline{x_1}, \ldots, \overline{x_d}]$ of $B$, and $B$ is a finitely generated $K[\overline{x_1}, \ldots, \overline{x_d}]$-supermodule. 

Let $\overline{D}$ denote $K[\overline{x_1}, \ldots, \overline{x_d}]$. Since $\overline{D}$ is a graded subalgebra of $B$, there is a generating set of $\overline{D}$-supermodule $B$ consisting of homogeneous elements, say $u_1\in B(k_1, l_1), \ldots , u_t\in B(k_t, l_t)$.
Choose also their representatives $v_1\in \mathfrak{N}_{k_1, l_1}, \ldots, v_t\in \mathfrak{N}_{k_t, l_t}$.

If $\mathrm{Ksdim}_1(B)=l$, then there are odd elements $u_{i_1}, \ldots, u_{i_l}$ such that 
\[\mathrm{Ann}_{\overline{D}}(u_{i_1} \ldots u_{i_l})=0.\]
It remains to show that the elements $v_{i_1}, \ldots, v_{i_l}$ form a system of odd parameters in $R$. If they do not, then
by Lemma \ref{Noether normalization for complete local} (2) there is a nonzero power series $f\in D$ such that $f v_{i_1}\ldots v_{i_l}=0$. 
Thus, if $g$ is the first nonzero homogeneous component of $f$, say of degree $m$, then $\overline{g}=g\pmod{\mathfrak{N}_{m+1, 0}\mathfrak{N}} \neq 0$ and  $\overline{g}u_{i_1}\ldots u_{i_l}=0,$
a contradiction. Proposition is proven.   
\end{proof}
Later we will show that the above ineaqulity is strict in general. But in some particular cases it turnes to be equality.

Assume that $\mathfrak{N}$ is a $\mathfrak{M}$-primary superideal from Proposition 
\ref{upper bound for odd} such that $\mathfrak{N}_1=R_1$. We call such ($\mathfrak{M}$-primary) superideal \emph{special}. 

Let $A$ be either a polynomial $K$-algebra or a power series $K$-algebra, freely generated by the elements $v_1, \ldots, v_d$.
Set $\mathfrak{v}=Av_1+\ldots +Av_d$. 
\begin{lm}\label{growth function}
If $M$ is a finitely generated $A$-module, then the function 
\[g_M : n\mapsto \dim M/\mathfrak{v}^{n+1}M\]
is a polynomial of degree at most $d$ for all sufficiently large $n$.
\end{lm} 
\begin{proof}
If $A$ is a polynomial algebra, then its completion in $\mathfrak{v}$-adic topology is isomorphic to $K[[v_1, \ldots, v_d]]$.
Moreover, $\widehat{\mathfrak{v}}=\widehat{A}\mathfrak{v}=\widehat{A}v_1+\ldots \widehat{A}v_d$ and Lemma \ref{completion as a functor} implies
that $\widehat{M}$ is a finitely generated $\widehat{A}$-module. Moreover, we have  
\[\widehat{M}/\widehat{\mathfrak{v}}^{n+1}\widehat{M}\simeq  M/\mathfrak{v}^{n+1}M\otimes_A\widehat{A}\simeq  M/\mathfrak{v}^{n+1}M\]
for any $n\geq 0$. Therefore, without loss of a generality one can assume that $A$ is a power series algebra. Then Corollary 11.5 from \cite{atmac}, or \cite{mats}, (12.C), conclude the proof.
\end{proof}	
\begin{lm}\label{a growth criterion}
$\mathrm{Ann}_A(M)=0$ if and only if the degree of $g_M$ is exactly $d$.	
\end{lm}
\begin{proof}
As it has been already observed, one can assume that $A$ is a power series algebra.
If $\mathrm{Ann}_A(M)\neq 0$, then $I_m=\mathrm{Ann}_A(m)\neq 0$ for any $m\in M$. Moreover, $\mathrm{Kdim}(A/I_m)=t\leq  d-1$ and  
there is a $\mathfrak{v}/I_m$-primary ideal $\mathfrak{t}$, generated by $t$ elements.
Thus the degree of $g_{Am}$ $\leq d-1$ (see Corollary 11.5 and Proposition 11.6, \cite{atmac}, or \cite{mats}, (12.C)).
Note that $M$ has a finite filtration whose quotients are monic $A$-modules. The induction on the length of this filtration and Proposition (12.D), \cite{mats},  imply that the degree of $g_M$ $\leq d-1$, hence the part "if".

Conversely, if the degree of $g_M$ $\leq d-1$, then by Proposition (12.D), \cite{mats}, the degree of $g_N$ $\leq d-1$ for each monic $A$-submodule $N$. Thus $\mathrm{Ann}_A(N)\neq 0$ and since $M$ is finitely generated, 
$\mathrm{Ann}_A(M)\neq 0$. Lemma is proven.
\end{proof}
\begin{theorem}\label{toward Hilbert polynomial}
If $\mathfrak{N}$ is a special $\mathfrak{M}$-primary superideal, then $\mathrm{Ksdim}(\mathsf{bgr}_{\mathfrak{N}}(R))=\mathrm{Ksdim}(R)$.
\end{theorem}
\begin{proof}
For arbitrary nonnegative integers $k, l$ we have
\[B(k, l)=\frac{\mathfrak{n}^kR_1^l + \mathfrak{n}^kR_1^{l+1}}{\mathfrak{n}^{k+1} R_1^l +\mathfrak{n}^k R_1^{l+2}+\mathfrak{n}^kR_1^{l+1}}\simeq\frac{\mathfrak{n}^kR_1^l}{\mathfrak{n}^{k+1} R_1^l +\mathfrak{n}^k R_1^{l+2}}.\]
As above, one can assume that $R$ is complete in the $\mathfrak{n}$-adic topology. Then $C=\mathsf{gr}_{I_R}(R)$ is complete in the 
$\overline{\mathfrak{n}}$-adic topology, where $\overline{\mathfrak{n}}=(\mathfrak{n}+R_1^2)/R_1^2$ is generated by the even parameters
$z_i=x_i +R_1^2, 1\leq i\leq d$. Besides, these elements freely generate a power series subalgebra $K[[z_1, \ldots , z_d]]$ of $C(0)=\overline{R}$ also, and $C$ is a finitely generated $K[[z_1, \ldots , z_d]]$-supermodule. Since $C(l)=C(1)^l$ for any positive integer $l$, Lemma \ref{Noether normalization for complete local}  infers $\mathrm{Ksdim}_1(C)=\max\{l\mid \mathrm{Ann}_D(C(l))=0\}$.

Consider a bigraded superalgebra $T=\mathsf{gr}_{\overline{\mathfrak{n}}}(C)$, whose (bigraded) components are 
\[T(k, l)=\overline{\mathfrak{n}}^k C(l)/\overline{\mathfrak{n}}^{k+1}C(l)\simeq \frac{\mathfrak{n}^k R_1^l +R_1^{l+2}}{\mathfrak{n}^{k+1} R_1^l +R_1^{l+2}}\simeq \frac{\mathfrak{n}^k R_1^l}{\mathfrak{n}^{k+1} R_1^l +(\mathfrak{n}^k R_1^l\cap R_1^{l+2})}.\]
We have $\mathfrak{n}^k R_1^{l+2}\subseteq \mathfrak{n}^k R_1^l\cap R_1^{l+2}$ for any $k, l\geq 0$, thus a canonical epimorphism $B\to T$ of bigraded superalgebras. As it has been already observed, the map $\overline{x_i}\mapsto t_i$ induces an isomorphism $\overline{D}\to K[t_1, \ldots, t_d]$ of polynomial rings. Since both $B$ and $T$
are finitely generated supermodules over $\overline{D}$ and $K[t_1, \ldots, t_d]$ respectiely, Lemma 4.4, \cite{zubmas}, implies
$\mathrm{Ksdim}_1(B)\geq\mathrm{Ksdim}_1(T)$. Therefore, if we will prove that $\mathrm{Ksdim}_1(T)=\mathrm{Ksdim}_1(C)$, then Proposition \ref{upper bound for odd} and Theorem  
\ref{grading_and_dimension_II} imply the statement.

Combining Lemma \ref{power series analog} with Lemma \ref{a growth criterion}, one sees that 
\[\mathrm{Ksdim}_1(C)=\max\{l \mid g_{C(l)} \ \mbox{has the degree} \ d\}, \] where $C$ is regarded as a $K[[z_1, \ldots, z_d]]$-supermodule. Recall that $T$ is a graded superalgebra with respect to the grading $T(l)=\oplus_{k\geq 0} T(k, l), l\geq 0$. Moreover, if $\mathfrak{v}$ 
denotes the ideal of $K[t_1, \ldots, t_d]$ generated by $t_1, \ldots , t_d$, then \[\mathfrak{v}^k T(l)=\oplus_{s\geq k} T(s, l) \ \mbox{and} \  \mathfrak{v}^k T(l)/\mathfrak{v}^{k+1}T(l)\simeq T(k, l), k, l\geq 0.\]
Furthermore, for any $l\geq 1$ there is $T(l)=T(1)^l$, and as above we have 
\[\mathrm{Ksdim}_1(T)=\max\{l \mid g_{T(l)} \ \mbox{has the degree} \ d .\} \] 
It remains to note that 
\[g_{C(l)}(n)=\sum_{0\leq k\leq n}\dim T(k, l)=g_{T(l)}(n)\]
for any nonnegative integer $n$, hence our theorem follows.	 
\end{proof} 
Let $\mathfrak{w}$ denote the ideal of $\overline{D}$ generated by the elements $\overline{x_1}, \ldots, \overline{x_d}$. Set $B(l)=\oplus_{k\geq 0} B(k, l), l\geq 0$. Then we have
\[\mathfrak{v}^k B(l)=\oplus_{s\geq k} B(s, l) \ \mbox{and} \  \mathfrak{v}^k B(l)/\mathfrak{v}^{k+1}B(l)\simeq B(k, l), k, l\geq 0,\]
similarly to the above. 
\begin{cor}\label{Hilbert polynomial}
We have $\mathrm{Ksdim}_1(R)=\max\{l\mid g_{B(l)} \ \mbox{has the degree} \ d\}$.
\end{cor}
\begin{proof}
Since $B\to T$ is a surjective morphism of bigraded superalgebras, for any $k, l\geq 0$ we have $g_{B(l)}(k)\geq g_{T(l)}(k)$.
Moreover, $B$ is a finitely generated $\overline{D}$-supermodule, hence the degree of each $g_{B(l)}$ is at most $d$. In particular, if the degree of some $g_{T(l)}$ is equal to $d$, then the degree of $g_{B(l)}$ is equal to $d$ as well. As above, we have $B(l)=B(1)^l$ for any $l\geq 1$. Thus  
\[\mathrm{sdim}_1(B)=\max\{l\mid g_{B(l)} \ \mbox{has the degree} \ d\}\leq \mathrm{sdim}_1(T)\leq\mathrm{sdim}_1(B).\] 
\end{proof}
This corollary allows us to define a \emph{Hilbert polynomial} of $R$ associated with a special $\mathfrak{M}$-primary superideal $\mathfrak{N}$ as
\[g_{\mathfrak{N}}(x, y)=\sum_{l\geq 0} g_{B(l)}(x)y^l.\]
Then Corollary \ref{Hilbert polynomial} can be reformulated as follows. Let $g_{\mathfrak{N}}(x, y)$ be a Hilbert polynomial of $R$ associated with a special $\mathfrak{M}$-primary superideal $\mathfrak{N}$, expressed as
\[g_{\mathfrak{N}}(x, y)=\sum_{l\geq 0} g_l(x)y^l . \]
Then the odd superdimension of $R$ is equal to the largest exponent $l$
such that the degree of $g_l(x)$ is equal to $\mathrm{sdim}_0(R)$.

\section{Fibres of morphisms of superschemes}

\begin{lm}\label{some exceptional case}
Let $A$ be a local regular super-ring. Let $z$ be a member of a minimal set of generators of $A_0$-module $A_1$. Then
$\mathrm{Ksdim}_1(A/Az)=\mathrm{Ksdim}_1(A)-1$. 
\end{lm}
\begin{proof}
Let $z=z_1, \ldots, z_s$ be a minimal set of generators of $A_0$-module $A_1$. Then $z_1, \ldots, z_s$ form an odd regular sequence in $A$, which is also the longest system of odd parameters of $A$ (see Proposition 5.2, \cite{zubmas}, or Corollary 3.3, \cite{sm}). Moreover, $A/Az$ is again regular, hence
\[\mathrm{Ksdim}_1(A/Az)= \dim_K((\overline{\mathfrak{P}}/\overline{\mathfrak{P}}^2)_1)=\dim_K((\mathfrak{P}/\mathfrak{P}^2)_1)-1=\mathrm{Ksdim}_1(A),\]
where $\mathfrak{P}$ is the maximal superideal of $A$ and $\overline{\mathfrak{P}}=(\mathfrak{P}+Az)/Az$.
\end{proof}
From now on a local super-ring contains a field, unless stated otherwise.
\begin{theorem}\label{well known ineaqulity}
Let $A\to B$ be a flat local morphism of local superalgebras and $A$ is regular. Then
\[\mathrm{Ksdim}_1(B)\geq \mathrm{Ksdim}_1(A)+\mathrm{Ksdim}_1(B/\mathfrak{P}B),\]
where $\mathfrak{P}$ is the maximal superideal of $A$. 
\end{theorem}
\begin{proof}
Set $\mathrm{Ksdim}(A)=r|s$. As it has been already observed, any minimal system of generators of $A_0$-module $A_1$ has the cardinality $s$, 
and it is the longest system of odd parameters of $A$ as well as the (longest) odd regular sequence of $A$.
Moreover, we have the exact sequence
\[0\to\sum_{1\leq i\leq s}Az_i=\mathrm{Ann}_{A}(z^{\underline{s}})\to A\to Az^{\underline{s}}\to 0,\]
where $A\to Az^{\underline{s}}$ is defined as $a\mapsto az^{\underline{s}}, a\in A$.
Since $B$ is flat over $A$, we have the exact sequence 
\[0\to B\otimes_A (\sum_{1\leq i\leq s}Az_i)\to B\otimes_A A\to B\otimes_A Az^{\underline{s}}\to 0,\]
which is naturally isomorphic to
\[0\to\sum_{1\leq i\leq s}Bz_i\to B\to Bz^{\underline{s}}\to 0,\]
hence $z_1, \ldots, z_s$ form an odd regular sequence in $B$ as well. 

Further, the superalgebra $B/Bz^{\underline{s}}\simeq A/Az^{\underline{s}}\otimes_A B$ is a flat $A/Az^{\underline{s}}$-supermodule. By Lemma \ref{factoring by parameters} (2) and Lemma \ref{some exceptional case}, we have 
\[\mathrm{Ksdim}_1(B/Bz^{\underline{s}})\geq \mathrm{Ksdim}_1(B)-s, \ \mathrm{Ksdim}_1(A/Az^{\underline{s}})=\mathrm{Ksdim}_1(A)-s=0 .\]
Since $Bz^{\underline{s}}\subseteq B\mathfrak{P}$ and $A/Az^{\underline{s}}\simeq\overline{A}$, it remains to prove our theorem in the case when $A$ is a regular ring. In other words, one has to show that
\[\mathrm{Ksdim}_1(B)\geq\mathrm{Ksdim}_1(B/\mathfrak{p}B),\]
where $\mathfrak{p}$ is the maximal ideal of $A$. The ideal $\mathfrak{p}$ is generated by a regular sequence $a_1, \ldots, a_r$, where $r=
\mathrm{Kdim}(A)$ (see Theorem 11.22 and Lemma 11.23, \cite{atmac}). 

Again, since $B$ is flat over $A$, the element $a_1$ is regular in $B$ and $B/Ba_1$ is flat over $A/Aa_1$. If we will show that
$\mathrm{Ksdim}_1(B)\geq\mathrm{Ksdim}_1(B/Ba_1)$, then Corollary 11.18, \cite{atmac}, combined with the induction on $r$, concludes the proof.

Theorem 19 (2), \cite{mats}, implies 
\[\mathrm{Kdim}(B_0)=\mathrm{Kdim}(\overline{B})= \mathrm{Kdim}(A)+\mathrm{Kdim}(B_0/B_0\mathfrak{p})=\mathrm{Kdim}(A)+\mathrm{Kdim}(\overline{B}/\overline{B}\mathfrak{p}).  \]
Let $b_1, \ldots, b_l$ form a system of parameters of $\overline{B}/\overline{B}\mathfrak{p}$. Then $b_1, \ldots, b_l, a_1, \ldots, a_r$ form a system of parameters of $\overline{B}$. Using Theorem 5.12, \cite{zubmas}, one can replace $B$ by $\widehat{B}$ and $B/Ba_1$ by $\widehat{B/Ba_1}\simeq \widehat{B}/\widehat{B}a_1$. Moreover, the elements $b_1, \ldots, b_l, a_1, \ldots, a_r$ remain parameters of $\overline{\widehat{B}}\simeq \widehat{\overline{B}}$. Besides, by Lemma \ref{completion as a functor} (2) $\widehat{B}$ is flat over $A$, whence $a_1$ is regular in $\widehat{B}$. 

If $B$ is complete, then Lemma \ref{power series analog} infers 
\[\mathrm{Ksdim}_1(B)=\max\{l\mid \mathrm{Ann}_D(B_1^l)=0\},\]
where $D=K[[b_1, \ldots, b_l, a_1, \ldots, a_r]]$ and $K$ is a coefficient field of $B$. Similarly, Corollary 11.18, \cite{atmac}, implies 
\[\mathrm{Kdim}(B_0/B_0a_1)=
\mathrm{Kdim}(B_0)-1 ,\] and thus
\[\mathrm{Ksdim}_1(B/Ba_1)=\max\{l\mid \mathrm{Ann}_{D'}(B_1^l/(B_1^l\cap Ba_1))=0\},\]
where $D'=K[[b_1, \ldots, b_l, a_2, \ldots, a_r]]$. 

Assume that an element $ 
x\in B_1^l$ is not annihilated by any nonzero element from $D'$ modulo $B_1^l\cap Ba_1$. Each element
$f\in D\setminus 0$ can be represented as $a_1^m(f_1 +a_1^n f_2), f_1\in D'\setminus 0, f_2\in D, m\geq 0, n>0$. 
The congruence \[(f_1 +a_1^n f_2)x \equiv f_1 x  \pmod{B_1^l\cap Ba_1}\] implies $(f_1 +a_1^n f_2)x\neq 0$, hence $fx\neq 0$, since $a_1$ is not zero divisor in $B$.  Thus $\mathrm{Ann}_D(x)=0$ and our theorem follows.
\end{proof}
Let $f : X\to Y$ be a morphism of superschemes. If $y\in Y^e$ is a point in the underlying topological space of $Y$, then similarly to
\cite{hart}, Chapter II, \S 3, one can define the \emph{fiber} of the morphism $f$ over the point $y$ to be the superscheme
\[X_y=X\times_Y \mathrm{SSpec}(K(y)),\]
where $K(y)=\mathcal{O}_{y, Y}/\mathfrak{m}_y$ and the closed embedding $\mathrm{SSpec}(K(y))\to Y$ is defined by the natural super-ring morphism
$\mathcal{O}(Y)\to\mathcal{O}_{y, Y}$. In general, the canonical projection $X_y\to X$ is not necessary immersion (cf. \cite{maszub2}, section 1.4). 
\begin{lm}\label{fiber as a space}
The topological space $X^e_y$ is naturally isomorphic to $f^{-1}(y)$.	
\end{lm} 
\begin{proof}
The topological space $X_y^e$ is identified with the underlying topological space of $(X_y)_{ev}$. By Lemma 2.3, \cite{maszub2},
$(X_y)_{ev}\simeq (X_{ev})_y$. It remains to refer to \cite{hart}, Chapter II, Exercise 3.10.  	
\end{proof}	

Recall that a superscheme $Z$ is a \emph{closed super-subscheme} of $X$, if there is a closed embedding $\iota : Z^e\to X^e$ such that the sheaf $\iota^*\mathcal{O}_Z$ is an epimorphic image of the sheaf $\mathcal{O}_X$.  
Equivalently, for any $z\in Z^e$ the induced local morphism $\mathcal{O}_{\iota(z), X}\to\mathcal{O}_{z, Z}$ is surjective. Note that a nonempty closed subset $Z$ of $X^e$ may have many possible closed super-subscheme structures. In particular, the odd dimension of $Z$ can vary with respect to the choice of such a structure.

Let $X$ be a superscheme of finite type over a field $K$. 
Let $Z$ be an irreducible component of $X^e$ equipped with a closed super-subscheme structure.  
\begin{lm}\label{certain point}
There is a point $z\in Z$ such that $\mathrm{sdim}_1(Z)\leq \mathrm{Ksdim}_1(\mathcal{O}_{z, X})$. 
\end{lm}
\begin{proof}
Let $U\simeq\mathrm{SSpec}(A)$ be an open affine super-subscheme of $X$ such that $U\cap Z\neq\emptyset$. Then $U\cap Z\simeq\mathrm{SSpec}(A/I)$ is an irreducible component of $U$, where $\sqrt{I_0}=\mathfrak{p}$ is a minimal prime ideal of $A_0$. The required point $z$ corresponds to the prime superideal $\mathfrak{P}=\mathfrak{p}\oplus A_1$. In fact, 
$\mathrm{sdim}_1(Z)=\mathrm{Ksdim}_1(A/I)$ and if the cosets of the elements $y_1, \ldots, y_k$ form a system of odd parameters of $A/I$, then 
\[(\mathrm{Ann}_{A_0}(y^{\underline{k}})+I_0 )/I_0\subseteq\mathrm{Ann}_{A_0/I_0}(\bar{y}^{\underline{k}})\subseteq \mathfrak{p}/I_0\] 
implies $\mathrm{Ann}_{A_0}(y^{\underline{k}})\subseteq\mathfrak{p}$. In particular, $S^{-1} A_1^k\neq 0$, where $S=A_0\setminus\mathfrak{p}$. Since 
the superalgebra $\mathcal{O}_{z, X}\simeq S^{-1}A$ is Artinian, Lemma \ref{over_Artinian} concludes the proof.
\end{proof}
Recall that a superscheme morphism $f : X\to Y$ id called \emph{flat} whenever for any $x\in X^e$ the induced local morphism
$\mathcal{O}_{f(x), Y}\to \mathcal{O}_{x, X}$ is flat.

It is known that for
arbitrary flat morphism of irreducible schemes $f : X\to Y$ each irreducible component of $X_y$, where $f^{-1}(y)\neq\emptyset$, has the dimension $\mathrm{dim}(X)-\mathrm{dim}(Y)$ (cf. \cite{hart}, Chapter III, Corollary 9.6).  In particular, $\mathrm{dim}(X_y)=\mathrm{dim}(X)-\mathrm{dim}(Y)$.
\begin{theorem}
Let $f : X\to Y$ be a flat morphism of irreducible superschemes. Let $y\in Y^e$ be a nonsingular point such that $f^{-1}(y)\neq\emptyset$ and let $Z$ be an irreducible component of $X_y^e$ equipped with a closed (irreducible) super-subscheme structure. Then   
\[\mathrm{sdim}_1(X)-\mathrm{sdim}_1(Y)\geq \mathrm{sdim}_1(Z).\]
\end{theorem}
\begin{proof}
Without loss of a generality one can assume that $X$ and $Y$ are affine, say $X=\mathrm{SSpec}(B)$ and $Y=\mathrm{SSpec}(A)$ (argue as in Lemma 2.1, \cite{maszub2}). Then $f$ is dual to a (flat) superalgebra morphism $A\to B$, the point $y$ corresponds to a prime superideal $\mathfrak{P}$ of $A$ and each point $x$ from $X_y^e$ corresponds to some prime superideal $\mathfrak{Q}$ of $B$ such that $A\cap \mathfrak{Q}=\mathfrak{P}$.	Moreover, the fiber $X_y\simeq \mathrm{SSpec}(B_{\mathfrak{P}}/(B\mathfrak{P})_{\mathfrak{P}})$ is also affine.

Note that Theorem \ref{well known ineaqulity} can be interpreted as follows. For any $x\in X_y^e$ there holds
\[\mathrm{sdim}_1(\mathcal{O}_{x, X})\geq \mathrm{sdim}_1(\mathcal{O}_{y, Y})+ \mathrm{sdim}_1(\mathcal{O}_{x, X_y}).\] 
In fact, we have \[\mathcal{O}_{x, X}\simeq B_{\mathfrak{Q}}, \mathcal{O}_{y, Y}\simeq A_{\mathfrak{P}}\] and 
\[\mathcal{O}_{x, X_y}\simeq ((B/B\mathfrak{P})_{\mathfrak{P}})_{(\mathfrak{Q}/B\mathfrak{P})_{\mathfrak{P}}}\simeq
B_{\mathfrak{Q}}/\mathfrak{P}B_{\mathfrak{Q}}.\]

Thus, if $x$ is as in Lemma \ref{certain point}, then the equalities $\mathrm{sdim}_1(\mathcal{O}_{x, X})=\mathrm{sdim}_1(X)$ and
$\mathrm{sdim}_1(\mathcal{O}_{y, Y})=\mathrm{sdim}_1(Y)$ imply our statement (see \cite{zubmas}, Remark 6.2). 
\end{proof}	
The following example shows that the inequality in Theorem \ref{well known ineaqulity} can be strict.
\begin{example}
In the example from the third section set $C=K$. Then $A$ is obviously a local Artinian superalgebra. Moreover, $R=A_{\pi}$ is a local Artinian superalgebra and a free $K[y]$-supermodule of the rank $\dim A$ as well, where $y=\Pi 1$. The superalgebra $K[y]$ is obviously regular and local with the maximal superideal $\mathfrak{P}=Ky$. Thus the natural embedding $K[y]\to R$ is flat and local, and we have
\[ \mathrm{Ksdim}_1(R) > \mathrm{Ksdim}_1(K[y])+\mathrm{Ksdim}_1(R/R\mathfrak{P}). \]    
\end{example} 

\end{document}